\documentclass[reqno,english]{amsart}
\usepackage{amsfonts,amsmath,latexsym,verbatim,amscd,mathrsfs,color,array}

\usepackage[utf8]{inputenc}
\usepackage[normalem]{ulem}
\usepackage{amsmath}
\usepackage{amsthm}
\usepackage{amsfonts}
\usepackage{amssymb}
\usepackage{indentfirst}
\usepackage{graphicx}
\usepackage{float}
\usepackage{amsmath,amssymb,amsthm,amsfonts,graphicx,color}
\usepackage{amssymb}
\usepackage{pdfsync}
\usepackage{epstopdf}
\usepackage[colorlinks=true]{hyperref}
\usepackage{subfigure}

\makeatletter
\def\@currentlabel{2.1}\label{e:dispaa}
\def\@currentlabel{2.21}\label{e:dispau}
\def\@currentlabel{2.22}\label{e:dispav}
\def\@currentlabel{2.23}\label{e:dispaw}
\def\@currentlabel{2.24}\label{e:dispax}
\def\theequation{\thesection.\@arabic\c@equation}
\makeatother
\let\oldbibliography\thebibliography
\renewcommand{\thebibliography}[1]{%
\oldbibliography{#1}%
\setlength{\itemsep}{0pt}%
}

\renewcommand{\theequation}{\thesection.\arabic{equation}}
\newtheorem{lemma}{Lemma}[section]

\newtheorem{definition}{Definition}[section]

\newtheorem{proposition}{Proposition}[section]
\newtheorem{corollary}{Corollary}[section]
\newtheorem{remark}{Remark}[section]
\newcommand{\bremark}{\begin{remark} \em}
\newcommand{\eremark}{\end{remark} }

\newtheorem{conjecture}{Conjecture}[section]
\newtheorem{theorem}{Theorem}[section]

\newcommand{\R}{{\mathbb R}}

\newcommand{\BE}{\begin{equation}}
\newcommand{\BEN}{\begin{equation*}}
\newcommand{\EE}{\end{equation}}
\newcommand{\EEN}{\end{equation*}}
\newcommand{\BL}{\begin{lemma}}
\newcommand{\EL}{\end{lemma}}
\newcommand{\BT}{\begin{theorem}}
\newcommand{\ET}{\end{theorem}}
\newcommand{\BP}{\begin{proposition}}
\newcommand{\EP}{\end{proposition}}
\newcommand{\BC}{\begin{corollary}}
\newcommand{\EC}{\end{corollary}}
\renewcommand{\Re}{\operatorname{Re}}
\renewcommand{\Im}{\operatorname{Im}}

\numberwithin{equation}{section}

\begin{document}


\title[Epstein Zeta functions vs Lennard-Jones Models]{\bf On Minima of Difference of Epstein Zeta Functions and Exact Solutions to Lennard-Jones Lattice Energy}

\author{Senping Luo}
\author{Juncheng Wei}

\address[S.~Luo]{School of Mathematics and statistics , Jiangxi Normal University, Nanchang, 330022, China; and Jiangxi applied mathematical research center.}
\address[J.~Wei]{Department of Mathematics, University of British Columbia, Vancouver, B.C., Canada, V6T 1Z2}

\email[S.~Luo]{luosp1989@163.com }

\email[J.~Wei]{jcwei@math.ubc.ca}

\begin{abstract}
Let $\zeta(s,z)=\sum_{(m,n)\in\mathbb{Z}^2\backslash\{0\}}\frac{(\Im(z))^s}{|mz+n|^{2s}}$ be the Eisenstein series/Epstein Zeta function.
Motivated by widely used Lennard-Jones potential
 \begin{equation}\aligned\nonumber
\mathcal{V}(|\cdot|^2):=4\varepsilon\Big(
(\frac{\sigma}{|\cdot|})^{12}-(\frac{\sigma}{|\cdot|})^{6}
\Big),
\endaligned\end{equation}
 in physics, in this paper, we consider the following lattice minimization problem
\begin{equation}\aligned\nonumber
\min_{z\in\mathbb{H}}\Big(\zeta(6,z)-b\zeta(3,z)\Big), \;\;b=\frac{1}{\sigma^6}
\endaligned\end{equation}
 and completely  classify  the minimizers  for all $b\in \R$. Our results resolve an open problem  in  Blanc-Lewin \cite{Bla2015}, and a conjecture by B\'etermin \cite{Bet2018}. Furthermore, our method of proofs works for general minimization problem
\begin{equation}\aligned\nonumber
\min_{z\in\mathbb{H}}\Big(\zeta(s_1,z)-b\zeta(s_2,z)\Big), \;\;s_1>s_2>1
\endaligned\end{equation}
which corresponds to general Lennard-Jones potential
 \begin{equation}\aligned\nonumber
\mathcal{V}(|\cdot|^2):=4\varepsilon\Big(
(\frac{\sigma}{|\cdot|})^{2s_1}-(\frac{\sigma}{|\cdot|})^{2s_2}
\Big),\;\;s_1>s_2>1.
\endaligned\end{equation}

\end{abstract}

\maketitle

\setcounter{equation}{0}

\section{Introduction and statement of main results}

Let $L$ be a two dimensional lattice, i.e., of the form $\Big({\mathbb Z}\vec{u}\oplus {\mathbb Z}\vec{v}\Big)$, where $\vec{u}$ and $\vec{v}$ are two independent two-dimensional vectors. Large classes of physical, chemical and mathematical problems can be formulated to the following minimization problem on lattice:
 \begin{equation}\aligned\label{EFL}
\min_L E_f(L ), \;\;\hbox{where}\;\;E_f(L):=\sum_{\mathbb{P}\in L \backslash\{0\}} f(|\mathbb{P}|^2),\; |\cdot|\;\hbox{is the Euclidean norm on}\;\mathbb{R}^2.
\endaligned\end{equation}
The summation ranges over all the lattice points except for the origin $0$ and the function $f$ denotes the potential of the system.
 The function
$E_f(L)$ denotes the limit energy per particle of the system under the background potential $f$ over a periodical lattice $L$. It arises in various physical problems. See  \cite{Bet2015,Bet2016,Bet2017,Bet2018a,Bet2018,Bet2019a,Bet2019,Bet2019AMP,Bet2020a,Betermin2021JPA, Serfaty2018,Bla2015,Radin1987,LW2022} and references therein. The functional $E_f(L)$ can be derived directly from crystallization problem when confining to the lattice-like atomic structure of crystals (\cite{Bet2017,Bet2018a,Bet2018,Bet2019a,Bet2019,Bet2019AMP,Bet2020a,Betermin2021JPA, Bla2015}). It is reasonable to consider lattice-like atomic structure since, as Weyl \cite{Weyl1949}(page 291) mentioned,  "{\em early in the history of crystallography the law of rational indices
was derived from the arrangement of the plane surfaces of crystals. It led to the hypothesis of the lattice-like atomic structure of crystals. This hypothesis, which explains the law of rational indices, has now been definitely confirmed by the Laue interference patterns, that are essentially X-ray photographs of crystals}". The function $E_f(L)$ also arises  in condensed matter physics. See Anderson \cite{Anderson}(pages 22-24).
The function $E_f(L)$ has close connections to sphere packing \cite{Radin1991} and Abrikosov vortex lattices (see e.g. \cite{Abr,Che2007,Sigal2018,Serfaty2010,Serfaty2012,Serfaty2014,Luo2019}).

The minimization problem \eqref{EFL} builds up intricate connection between analytical number theory and statistical physics (see e.g. \cite{Serfaty2010,Serfaty2012,Serfaty2018,Sigal2018,Sar2006,Bla2015,Bla2022,Cohen,Luo2022,LW2021,Osg1988,Sch2011,Number}). In analytical number theory, one looks for where the minimizer is achieved; while in statistical physics, one asks why there appears hexagonal lattice or square lattice or rhombic lattice or mix of these several lattice shapes. It turns out these two considerations coincide. More precisely, let $ z\in \mathbb{H}:=\{z= x+ i y\;\hbox{or}\;(x,y)\in\mathbb{C}: y>0\}$. For  lattice $L$ with  unit cell area we can  use the parametrization $L =\sqrt{\frac{1}{\Im(z)}}\Big({\mathbb Z}\oplus z{\mathbb Z}\Big)$   where $z \in \mathbb{H}$. When
 $$ f(|\cdot|^2)=e^{- \pi\alpha |\cdot|^2},\;\alpha>0\;\;\hbox{and}\;\;f(|\cdot|^2)=\frac{1}{|\cdot|^{2s}},\;s>1$$
  denoting the Gaussian and Riesz potential respectively, the corresponding limit energy per particle $E_f (L)$ becomes the theta function and Epstein Zeta function respectively, namely,
 \begin{equation}\aligned
 \label{thetas}
\theta (\alpha, z):&=\sum_{\mathbb{P}\in L} e^{- \pi\alpha |\mathbb{P}|^2}=\sum_{(m,n)\in\mathbb{Z}^2} e^{-\pi\alpha \frac{\pi }{\Im(z) }|mz+n|^2},\\
\zeta(s,z):&=\sum_{\mathbb{P}\in L\setminus\{0\}}\frac{1}{|\mathbb{P}|^{2s}}=\sum_{(m,n)\in\mathbb{Z}^2\backslash\{0\}}\frac{(\Im(z))^s}{|mz+n|^{2s}}.
\endaligned\end{equation}
Number theorist Rankin \cite{Ran1953} initiated the study of minimizer of $\zeta(s,z)$ for some range of $s$, and followed by Cassels \cite{Cas1959}, Diananda \cite{Dia1964} and Ennola  \cite{Enn1964a, Enn1964b}. They proved that the point $z=e^{i\frac{\pi}{3}}$ achieves the minimizer (up to rotation and translation). Montgomery \cite{Mon1988} proved that $\theta (\alpha, z)$ still achieves the same minimizer point for all $\alpha>0$ (thus recovering the results of \cite{Ran1953,Cas1959,Enn1964a,Enn1964b}). Montgomery's Theorem \cite{Mon1988} has profound applications in mathematical physics, see e.g. \cite{Betermin2021b, Serfaty2018}. When
$$
f(\cdot)=-\log|\cdot|
$$
denotes the Coulomb potential, it is implicitly proved by \cite{Serfaty2012} the $z=e^{i\frac{\pi}{3}}$ still achieves the minimizer by
their constructed Coulombian renormalized energy via a regularized procedure. It was extended to a two component Coulombian competing system
by \cite{Luo2019}.

The most physically related and widely used potential is the Lennard-Jones potential.
It takes the form (see e.g. \cite{Hansen1969,Verlet1967,John1993})
 \begin{equation}\aligned\label{LJ}
f(|\cdot|^2):=4\varepsilon\Big(
(\frac{\sigma}{|\cdot|})^{12}-(\frac{\sigma}{|\cdot|})^{6}
\Big),
\endaligned\end{equation}
where $|\cdot|$ is the distance between two interacting particles, $\varepsilon$  is the depth of the potential well (usually referred to as 'dispersion energy'), $\sigma$ can be taken as the size of the repulsive core or the
effective diameter of the hard sphere atom (interaction diameter). This is a  fundamental model in physics if  both attractions and repulsive interaction forces are present. It is considered an archetype model for simple yet realistic intermolecular interactions and is not only of fundamental importance in computational chemistry and soft-matter physics, but also for the modeling of real substances. We refer to \cite{Betermin2021a} for background and references for Lennard-Jones type potential.

The limit energy per particle \eqref{EFL} under Lennard-Jones potential \eqref{LJ} is given by

 \begin{equation}\aligned\label{VLJa}
\mathcal{E}_{LJ}=2\varepsilon\sigma^{12}
\Big(
\zeta (6,z)-b\zeta (3,z)
\Big), \;\;b=\frac{1}{\sigma^6}.
\endaligned\end{equation}
The problem of optimal lattice shape to the Lennard-Jones model \eqref{VLJa} is then reduced to finding the minimizers of the following minimization problem
 \begin{equation}\aligned\label{VLJ}
\min_{z\in \mathbb{H}}
\Big(
\zeta (6,z)-b\zeta (3,z)
\Big), \;\;b\in \R
\endaligned\end{equation}
where $\zeta (\alpha,z)$ is the zeta-function defined in (\ref{thetas}).

B\'etermin and his collaborators \cite{Bet2016,Bet2015,Bet2018,Betermin2021a} initiated the study of lattice minimization problem for  the Lennard-Jones model \eqref{VLJ}.
(Note that there are different notations between B\'etermin's and ours here, while they are essentially  the same (up to some constants).)
B\'etermin made the following conjecture in \cite{Bet2018}, and also mentioned in \cite{Bet2019,Betermin2021JPA}.
(Note that in the survey of crystallization problems by Blanc-Lewin \cite{Bla2015}(pages 268-269), they pointed out "{\em the energy minimizer to physically related \eqref{VLJ} is still unknown".})

\begin{conjecture}[\cite{Bet2018}, page 4003]\label{ConB} The solutions to Lennard-Jones model \eqref{VLJ} are as follows:
\begin{itemize}
  \item [(1)] For $\frac{\pi}{(120)^{1/3}}<A<A_{BZ}\approx1.138$, the minimizer is triangular.
  \item [(2)] For $A_{BZ}<A<A_1\approx1.143$, the minimizer is rhombic lattice. More precisely it covers continuously and
  monotonically the interval of angles $[76.43,90).$
  \item [(3)] For $A_1<A<A_2\approx1.268$, the minimizer is the unique minimizer.
  \item [(4)] For $A>A_2$,  the minimizer is a rectangular lattice.
\end{itemize}
\end{conjecture}
Here B\'etermin's parameter $A$ (denoting the area/density of lattice) and the parameter $b$(scaling of the interaction diameter) in \eqref{VLJ} has the following relation:
\begin{equation}\aligned\label{bAAA}
b=2A^3.
\endaligned\end{equation}

A numerical simulation to support B\'etermin's Conjecture \ref{ConB} can be found in \cite{Tra2019}.
Note that B\'etermin \cite{Bet2016,Bet2018} rigorously proved that the minimizer is triangular for $A<\frac{\pi}{(120)^{1/3}}$ by skillfully using the result of Montgomery \cite{Mon1988}. Montgomery's Theorem \cite{Mon1988} and B\'etermin's method(see Theorem 2.9 in \cite{Betermin2021AHP} for a general criterion) does not work for the proof in cases when minimizer is no-longer triangular and cannot find optimal bound of the minimizer is triangular as pointed out by B\'etermin \cite{Bet2016,Bet2018}. This suggests the need of a new method and framework to attack the minimization problem (\ref{VLJ}). In addition, B\'etermin \cite{Bet2018} showed that related to case $[(4)]$ in Conjecture \ref{ConB}, when $A\rightarrow+\infty$, the minimizer becomes more and more thin and rectangular.

In this paper, we consider Lennard-Jones model \eqref{VLJ} and give complete solutions to \eqref{VLJ}
and then answer positively and affirmatively to Conjecture \ref{ConB} and open problem by Blanc-Lewin \cite{Bla2015}(pages 268-269).

We state our main results in the following
\begin{theorem}\label{Th1} The complete solutions to the minimization problem with parameter $b\in \R$
\begin{equation}\aligned\nonumber
\min_{z\in\mathbb{H}}\Big(\zeta(6,z)-b\zeta(3,z)
\Big),  b\in\R,
\endaligned\end{equation}
are as follows: there exists three thresholds $0<b_1<b_2<b_3<\infty$ such that
\begin{itemize}
  \item [(1)] for $b\in(-\infty, b_1)$, the minimizer is $e^{i\frac{\pi}{3}}$, corresponding to  hexagonal$/$triangular lattice;
  \item [(2)] for $b=b_1$, the minimizer is $e^{i\frac{\pi}{3}}$ or $e^{i\theta_{b_1}}$, corresponding to hexagonal lattice or rhombic lattice with particular angle;
  \item [(3)] for $b\in(b_1, b_2)$, the minimizer is $e^{i\theta_{b}}$, where $\theta_b\in(\theta_{b_1},\frac{\pi}{2})$ and $\frac{d\theta_b}{db}>0$, corresponding  rhombic lattice with the angle ranging continuously and
  monotonically the interval of $(\theta_{b_1},\frac{\pi}{2})$;
  \item [(4)] for $b\in[b_2,b_3]$, the minimizer is $i$, corresponding square lattice;
  \item [(5)] for $b\in(b_3,\infty)$, the minimizer is $iy_b, y_b\in(1,\infty)$, and $\frac{dy_b}{db}>0, \lim_{b\rightarrow\infty}y_b=\infty$. Furthermore
  $$
  \lim_{b\rightarrow\infty}\frac{y_b}{\sqrt[3]{\frac{\xi(6)}{2\xi(12)}\cdot b}}=1.
  $$
 This minimizer corresponds to  rectangular lattice whose  ratio of longer side to shorter side is increasing from 1 to $\infty$ with a asymptotic rate
$\sqrt[3]{\frac{\xi(6)}{2\xi(12)}\cdot b}$.
\end{itemize}

Here $\xi$ is the Riemann Zeta function and defined by
\begin{equation}\aligned\nonumber
\xi(s):=\sum_{n=1}^\infty \frac{1}{n^s}.
\endaligned\end{equation}

Numerically,
\begin{equation}\aligned\nonumber
b_1=2.9465\cdots, b_2=2.9866\cdots, b_3=4.0774\cdots,
\theta_{b_1}=1.3347\cdots.
\endaligned\end{equation}
Analytically,
\begin{equation}\aligned\nonumber
b_2:&=\frac{\frac{\partial\zeta(6,\frac{1}{2}+iy)}{\partial y}}
{\frac{\partial\zeta(3,\frac{1}{2}+iy)}{\partial y}}\mid_{y=\frac{1}{2}},\\
b_3:&=\frac{\frac{\partial\zeta(6,iy)}{\partial y}}
{\frac{\partial\zeta(3,iy)}{\partial y}}\mid_{y=1},
\endaligned\end{equation}
and $b_1, \theta_{b_1}$ are determined by an equation \eqref{fbbb} and Lemma \ref{LemmaH4}.
$y_b$ is the unique solution of the equation
\begin{equation}\aligned\label{yb100}
b=\frac{\frac{\partial\zeta(6,iy)}{\partial y}}
{\frac{\partial\zeta(3,iy)}{\partial y}}\;\;\hbox{for}\;\; y\geq1.
\endaligned\end{equation}

In summary, it admits five cases which can be summarized as follows
\begin{equation}\aligned\nonumber
\min_{z\in\mathbb{H}}\Big(\zeta(6,z)-b\zeta(3,z)\Big)\;\;\hbox{is achived at}\;\;\begin{cases}
e^{i\frac{\pi}{3}},\;\;\;\;\;\;\;\;\;\;\;\;\;\;\;\;\;\;\;\;\;\;\;\;\;\;\;\;\;\;\;\; \;\;\;\;\;\;\;\;\;\;\;\;\;\;\;\;\;\;\;\;\;\;\;\;\;\;\;b\in (-\infty,b_1);\\
e^{i\frac{\pi}{3}}, \;\;\hbox{or}\;\;e^{i\theta_{b_{1}}},\theta_{b_1}=1.3347\cdots,\;\;\;\;\;\;\;\;\;\;\;\;\;\;\;\;\;\;\;\;\;\;\;\;\;\; b=b_1;\\
e^{i\theta_{b}},\theta_{b}\in(\theta_{b_1},\frac{\pi}{2}),\;\;\;\;\;\;\;\;\;\;\;\;\;\;\;\;\;\;\;\;\;\;\;\;\;\;\;\;\;\;\;\;\;\;\;\;\;\;\;\;\; b\in (b_1,b_2);\\
i,\;\;\;\;\;\;\;\;\;\;\;\;\;\;\;\;\;\;\;\;\;\;\;\;\;\;\;\;\;\;\;\;\;\;\;\;\;\;\;\;\;\;\;\;\;\;\;\;\;\;\;\;\;\;\;\;\;\;\;\;\;\;\;\;\;\;\; b\in [b_2,b_3];\\
iy_b, y_b>1, (y_b\rightarrow\sqrt[3]{\frac{\xi(6)}{2\xi(12)}\cdot b},\hbox{as}\; b\rightarrow\infty),\;\; \;\;\;\;\;\;b\in (b_3,\infty).
\end{cases}
\endaligned\end{equation}
\end{theorem}

{\bf
\begin{table}[!htbp]\label{TableB}
\caption{{\bf Minimizers of $\Big(\zeta(6,z)-b\zeta(3,z)\Big)$ for $b\geq b_3$: numerical aspect by the minimizer formula \eqref{yb100}.}}
\label{demo1}
\centering
\begin{tabular}{|c|c|c|c|}

\hline

values of $b$  & Lattice shape & Values of $b$ & Lattice shape \\

\hline

$b\in[b_3,\infty)$ &  Rectangular lattice=$iy_b$ & $b\in[b_3,\infty)$ & Rectangular lattice=$iy_b$\\

\hline
$b=b_3$ & $iy_b=i$ & $b=b_3$ &  $iy_b=i$\\
\hline
$b=5$ &  $iy_b=i1.249803800\cdots$ & $b=100$ &  $iy_b=i3.703484053\cdots$\\
\hline
$b=6$ &  $iy_b=i1.372027647\cdots$ & $b=200$ &  $iy_b=i4.667323033\cdots$\\
\hline
$b=7$ &  $iy_b=i1.467520869\cdots$ & $b=300$ &  $iy_b=i5.343067509\cdots$\\
\hline
$b=8$ &  $iy_b=i1.548848505\cdots$ & $b=400$ &  $iy_b=i5.880944029\cdots$\\
\hline
$b=9$ &  $iy_b=i1.620862121\cdots$ & $b=500$ &  $iy_b=i6.335129562\cdots$\\

\hline
$b=10$ &  $iy_b=i1.686088356\cdots$ & $b=600$ &  $iy_b=i6.732125854\cdots$\\

\hline
$b=10^3$ &  $iy_b=i7.981906081\cdots$ & $b=10^7$ &  $iy_b=i171.9663699\cdots$\\

\hline
$b=10^4$ &  $iy_b=i17.196633949\cdots$ & $b=10^8$ &  $iy_b=i370.4903128\cdots$\\
\hline
$b=10^5$ &  $iy_b=i37.04903113\cdots$ & $b=10^9$ &  $iy_b=i798.1971821\cdots$\\
$b=10^6$ &  $iy_b=i79.81971820\cdots$ & $b=10^{10}$ &  $iy_b=i1719.663699\cdots$\\

\hline

\end{tabular}
\end{table}
}

\begin{figure}
\centering
 \includegraphics[scale=0.58]{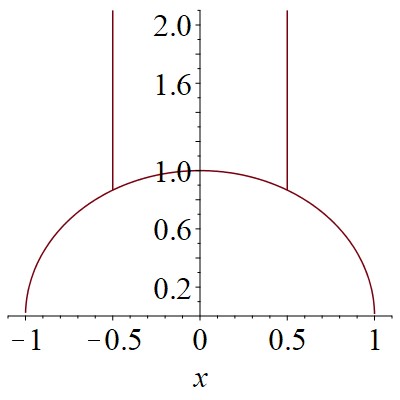}\includegraphics[scale=0.58]{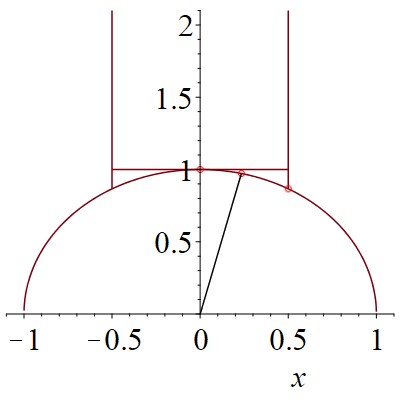}
 \caption{Location of the fundamental region and its subregions; hexagonal point($e^{i\frac{\pi}{3}}$), square point($i$), and a particular rhombic point($e^{i\theta_{b_1}}$).}
\label{PFFF}
\end{figure}

\begin{remark} By the relation in \eqref{bAAA}, we give a complete proof to Conjecture \ref{ConB}.
\end{remark}

\begin{remark}
We find that there admits {\em two} minimizers for the first threshold $b=b_1$ up to rotations and translations. This is a new phenomenon mathematically
and we do not know any potential  physical implications.
\end{remark}

\begin{remark}
Our method of proofs can be easily adopted to treat the more general Lenard-Jones potential cases:
\begin{equation}\aligned\label{LJm}
f(|\cdot|^2):=4\varepsilon\Big(
(\frac{\sigma}{|\cdot|})^{2s_1}-(\frac{\sigma}{|\cdot|})^{2s_2}
\Big), s_1>s_2>1.
\endaligned\end{equation}
\end{remark}

 Throughout this paper we denote
\begin{equation}
\nonumber
\mathbb{H}:=\{z= x+ i y\in\mathbb{C}: y>0\},
\end{equation}

\begin{equation}\aligned\nonumber
\mathcal{D}_{\mathcal{G}}:=\{
z\in\mathbb{H}: |z|>1,\; 0<x<\frac{1}{2}
\}.
\endaligned\end{equation}
The boundaries of $ \mathcal{D}_{\mathcal{G}}$ are divided into two pieces
\begin{equation}\aligned\nonumber
\Gamma_a:&=\{
z\in\mathbb{H}: \Re(z)=0,\; \Im(z)\geq1
\},\\
\Gamma_b:&=\{
z\in\mathbb{H}: |z|=1,\; \Im(z)\in[\frac{\sqrt3}{2},1]
\},
\endaligned\end{equation}
\begin{equation}\aligned\nonumber
\Gamma=\Gamma_a\cup\Gamma_b.
\endaligned\end{equation}
See Picture \ref{PFFF}.

In the rest of this paper, we prove Theorem \ref{Th1}. We introduce the main ideas of the proof and the organization of the paper.
B\'etermin (\cite{Bet2018}, page 4002) suggests a strategy to prove his Conjecture \ref{ConB}. For $b \geq 3.062$, we partially inspired by step 1 of
the strategy with the opposite direction in proving \eqref{reduction2}(Theorem \ref{ThA}). See details in Section 3.  For $b \leq 3.06$, we estimate the mixed order second order derivatives
$$ \frac{\partial^2}{\partial y^2} \Big(
\zeta(6,z)-b\zeta(3,z)
\Big);\;\; \frac{\partial^2}{\partial y \partial x} \Big(
\zeta(6,z)-b\zeta(3,z)
\Big)$$
and reduce the minimization problem to $\frac{1}{4}-$arc $\Gamma_b$. (See Propositions \ref{Lemma101}-\ref{Lemma102}).)  This kind of estimate seems to be the first in the study of lattice minimization problems. We believe that there may be more potential applications of these estimates.

In Section 2, we present some basic properties of the function $\Big(
\zeta(6,z)-b\zeta(3,z)
\Big)$ including its group invariance. As a result of group invariance we deduce that
\begin{equation}\aligned\label{reduction1}
\min_{z\in \mathbb{H}}\Big(
\zeta(6,z)-b\zeta(3,z)
\Big)=\min_{z\in \overline{\mathcal{D}_{\mathcal{G}}}} \Big(
\zeta(6,z)-b\zeta(3,z)
\Big).
\endaligned\end{equation}

In Section 3, we prove that
\begin{equation}\aligned\label{reduction2}
\hbox{for}\;\;b\geq3.062,\;\;\min_{z\in \mathbb{H}} \Big(
\zeta(6,z)-b\zeta(3,z)
\Big)
=\min_{z\in\Gamma} \Big(
\zeta(6,z)-b\zeta(3,z)
\Big).
\endaligned\end{equation}
This follows from the monotonicity in  the horizontal $x-$direction (see Proposition \ref{Lemma81})
\begin{equation}\aligned\label{mon1}
\frac{\partial}{\partial x} \Big(
\zeta(6,z)-b\zeta(3,z)
\Big)\geq0, \hbox{for}\; z\in \overline{\mathcal{D}_{\mathcal{G}}}.
\endaligned\end{equation}

In Section 4, we prove that
\begin{equation}\aligned\label{reduction4}
\hbox{for}\;\;b\leq3.06,\;\;\min_{z\in \mathbb{H}}\Big(
\zeta(6,z)-b\zeta(3,z)
\Big)
=\min_{z\in\Gamma}\Big(
\zeta(6,z)-b\zeta(3,z)
\Big).
\endaligned\end{equation}
This follows from positivity of two second order   derivatives  (see Propositions \ref{Lemma101}-\ref{Lemma102}).

In Section 5, we prove that
\begin{equation}\aligned\nonumber
\hbox{for}\;\;b\in(3.06,3.062),\;\;\min_{z\in\mathbb{H}}\Big(
\zeta(6,z)-b\zeta(3,z)
\Big)
\;\;\hbox{is achieved at }i.
\endaligned\end{equation}
By the results in Sections 3-5, we conclude that
\begin{equation}\aligned\nonumber
\hbox{for}\;\;b\in \R,\;\;\min_{z\in \mathbb{H}}\Big(
\zeta(6,z)-b\zeta(3,z)
\Big)
=\min_{z\in\Gamma}\Big(
\zeta(6,z)-b\zeta(3,z)
\Big).
\endaligned\end{equation}
This partially answers a question by B\'etermin \cite{Betermin2021JPA}, where he remarked that {\em in the fixed density case, the set of lattices is much larger
and this is not clear why such minimizer must be in the boundary of the fundamental domain for 2d lattices$($i.e., must be either rhombic or orthorhombic$)$}(Remark 4.3, page 15).

In Sections 6 and 7, we locate exactly where the minimizers should be for $\Big(
\zeta(6,z)-b\zeta(3,z)
\Big)$ when $z\in\Gamma_a$ and when $z\in\Gamma_b$ respectively. Finally, we give the proof of Theorem \ref{Th1} in Section 8.

\section{Preliminaries}
\setcounter{equation}{0}
In this section we present the relation between minimizers and lattice shapes, and some group invariance of $\Big(\zeta(6,z)-b\zeta(3,z)\Big)$
, followed by its corresponding fundamental region.

We first fix the parametrization of a lattice in $\mathbb{R}^2$ (see e.g. \cite{Bet2015,Bet2016,Bet2017,Bet2018a,Bet2018,Bet2019a,Bet2019}). Let the lattice $\Lambda=\mathbb{Z}\mathbf{a}_1\oplus\mathbb{Z}\mathbf{a}_2$ be spanned by two vectors $\mathbf{a}_1, \mathbf{a}_2$. For convenience,
we fix the area of the lattice to be $1$, namely, $|\mathbf{a}_1\wedge\mathbf{a}_2|=1$.
Up to rotation and translation, one can set $\mathbf{a}_1=\frac{1}{\sqrt y}(1,0)$, and $\mathbf{a}_2=\frac{1}{\sqrt y}(x,y)$, where $y>0$.
In this way, $|\mathbf{a}_1\wedge\mathbf{a}_2|=1$, $\mathbf{a}_2 = z\mathbf{a}_1$ where $z=x+iy \in \mathbb{H}:=\{z= x+ i y\in\mathbb{C}: y>0\}$, and $\Lambda =\sqrt{\frac{1}{y}}\Big({\mathbb Z}\oplus z{\mathbb Z}\Big)$. In particular,
$z=e^{i\frac{\pi}{3}}$ corresponds to hexagonal lattice, $z=i$ corresponds to square lattice,$z=e^{i\theta}$, $\theta\in(0,\frac{\pi}{2})$ corresponds to rhombic lattice, and $z=iy, y\geq 1$ corresponds to rectangular lattice. (If $y>1$ it corresponds to strict rectangular lattice.)

Let
$
\mathbb{H}
$
 denote the upper half plane and  $\mathcal{S} $ denote the modular group
\begin{equation}\aligned\label{modular}
\mathcal{S}:=SL_2(\mathbb{Z})=\{
\left(
  \begin{array}{cc}
    a & b \\
    c & d \\
  \end{array}
\right), ad-bc=1, a, b, c, d\in\mathbb{Z}
\}.
\endaligned\end{equation}

We use the following definition of fundamental domain which is slightly different from the classical definition (see \cite{Mon1988}):
\begin{definition} [page 108, \cite{Eva1973}]
The fundamental domain associated to group $G$ is a connected domain $\mathcal{D}$ satisfies
\begin{itemize}
  \item For any $z\in\mathbb{H}$, there exists an element $\pi\in G$ such that $\pi(z)\in\overline{\mathcal{D}}$;
  \item Suppose $z_1,z_2\in\mathcal{D}$ and $\pi(z_1)=z_2$ for some $\pi\in G$, then $z_1=z_2$ and $\pi=\pm Id$.
\end{itemize}
\end{definition}

By Definition 1, the fundamental domain associated to modular group $\mathcal{S}$ is
\begin{equation}\aligned\label{Fd1}
\mathcal{D}_{\mathcal{S}}:=\{
z\in\mathbb{H}: |z|>1,\; -\frac{1}{2}<x<\frac{1}{2}
\}
\endaligned\end{equation}
which is open.  Note that the fundamental domain can be open. (See [page 30, \cite{Apo1976}].)

Next we introduce another group related  to the functionals $\zeta(s;z)$. The generators of the group are given by
\begin{equation}\aligned\label{GroupG1}
\mathcal{G}: \hbox{the group generated by} \;\;\tau\mapsto -\frac{1}{\tau},\;\; \tau\mapsto \tau+1,\;\;\tau\mapsto -\overline{\tau}.
\endaligned\end{equation}

It is easy to see that
the fundamental domain associated to group $\mathcal{G}$ denoted by $\mathcal{D}_{\mathcal{G}}$ is
\begin{equation}\aligned\label{Fd3}
\mathcal{D}_{\mathcal{G}}:=\{
z\in\mathbb{H}: |z|>1,\; 0<x<\frac{1}{2}
\}.
\endaligned\end{equation}

The following lemma characterizes the fundamental symmetries of the theta functions $\theta (s; z)$. The proof is easy so we omit it.
\begin{lemma}\label{G111} For any $s>0$, any $\gamma\in \mathcal{G}$ and $z\in\mathbb{H}$,
$\ \zeta (s; \gamma(z))=\zeta (s;z)$.
\end{lemma}

Let
\begin{equation}\aligned\nonumber
\mathcal{W}_b(\alpha; z):=\zeta(\alpha;z)-b\zeta(2\alpha;z).
\endaligned\end{equation}

\begin{lemma}\label{Geee}  For any $\alpha>0$, any $\gamma\in \mathcal{G}$ and $z\in\mathbb{H}$,
$\mathcal{W}_b(\alpha; \gamma(z))=\mathcal{W}_b (\alpha;z)$.
\end{lemma}


\section{The Horizontal monotonicity: the cases $b\geq3.062$ }

Recall that the partial boundary is defined as follows
\begin{equation}\aligned\nonumber
\Gamma_a&=\{
z\in\mathbb{H}: \Re(z)=0,\; \Im(z)\geq1
\};\\
\Gamma_b&=\{
z\in\mathbb{H}: |z|=1,\; \Im(z)\in[\frac{\sqrt3}{2},1]
\},
\endaligned\end{equation}
\begin{equation}\aligned\nonumber
\Gamma=\Gamma_a\cup\Gamma_b.
\endaligned\end{equation}
Note that by group invariance(Lemma \ref{Geee}) and its fundamental region(\eqref{Fd3}), one has
\begin{lemma}\label{FRegion} For all $b\in\R$,
\begin{equation}\aligned\nonumber
\min_{z\in \mathbb{H}}\Big(
\zeta(6,z)-b\zeta(3,z)
\Big)=\min_{z\in \overline{\mathcal{D}_{\mathcal{G}}}}\Big(
\zeta(6,z)-b\zeta(3,z)
\Big)
.
\endaligned\end{equation}
\end{lemma}
By Lemma \ref{FRegion}, one can reduce the finding of  the minimizer of $\Big(
\zeta(6,z)-b\zeta(3,z)
\Big)$ from $z\in \mathbb{H}$ to $z\in \overline{\mathcal{D}_{\mathcal{G}}}$.

In this section, we aim to prove that
\begin{theorem}\label{ThA} \begin{itemize}
                             \item For $b\geq3.062$, then
\begin{equation}\aligned\nonumber
\min_{z\in \mathbb{H}}\Big(
\zeta(6,z)-b\zeta(3,z)
\Big)=\min_{z\in \overline{\mathcal{D}_{\mathcal{G}}}}\Big(
\zeta(6,z)-b\zeta(3,z)
\Big)
=\min_{z\in\Gamma}\Big(
\zeta(6,z)-b\zeta(3,z)
\Big).
\endaligned\end{equation}
                             \item For $b\geq3$, then
\begin{equation}\aligned\nonumber
\min_{z\in \overline{\mathcal{D}_{\mathcal{G}}}\cap\{y\geq1\}}\Big(
\zeta(6,z)-b\zeta(3,z)
\Big)
=\min_{z\in\Gamma_a}\Big(
\zeta(6,z)-b\zeta(3,z)
\Big).
\endaligned\end{equation}
                           \end{itemize}

\end{theorem}
See Picture \ref{PFFF} for the geometric shapes of $\overline{\mathcal{D}_{\mathcal{G}}}$, $\overline{\mathcal{D}_{\mathcal{G}}}\cap\{y\geq1\}$, $\Gamma_a$ and $\Gamma$.
The proof of Theorem \ref{ThA} is based on the following horizontal monotonicity:

\begin{proposition}\label{Lemma81}
 \begin{itemize}
                                     \item For $b\geq3.062$, then
\begin{equation}\aligned\nonumber
\frac{\partial}{\partial x}\Big(
\zeta(6,z)-b\zeta(3,z)
\Big)
\geq0, \hbox{for}\; z\in \overline{\mathcal{D}_{\mathcal{G}}}.
\endaligned\end{equation}
                                     \item For $b\geq3$, then
\begin{equation}\aligned\nonumber
\frac{\partial}{\partial x}\Big(
\zeta(6,z)-b\zeta(3,z)
\Big)
\geq0, \hbox{for}\; z\in \overline{\mathcal{D}_{\mathcal{G}}}\cap\{y\geq1\}.
\endaligned\end{equation}
                                   \end{itemize}

\end{proposition}

In the rest of this Section, we prove Proposition \ref{Lemma81}.
To prove Proposition \ref{Lemma81}, we first  introduce the Chowla-Selberg formula for $\zeta$ functions and related Bessel functions.
The following theorem is a classical result in number theory.  (See the monograph \cite{Cohen2007}(page 211).)
\begin{theorem}[Chowla-Selberg formula, \cite{Chowla1967,Chowla1949}]\label{ThCS}The Fourier expansion of the Epstein Zeta function is given by
 \begin{equation}\aligned\label{F31}
\zeta(s,z)=&2\xi(2s) y^s+2\sqrt{\pi}\frac{\Gamma(s-\frac{1}{2})}{\Gamma(s)}\xi(2s-1)y^{1-s}\\
&\;\;\;\;+\frac{8\pi^s}{\Gamma(s)}\sqrt{y}\sum_{n=1}^\infty \sigma_{s-\frac{1}{2}}(n)\cdot K_{s-\frac{1}{2}}(2mn\pi y)\cdot\cos(2\pi mn x).
\endaligned\end{equation}
Here $\xi$ is the Riemann Zeta function, $\Gamma$ is the Gamma function, and explicitly
 \begin{equation}\aligned\nonumber
\xi(s)=\sum_{m=1}^\infty\frac{1}{m^s},\;\; \Gamma(s)=\int_0^\infty e^{-\frac{1}{t}}\frac{1}{t^{1+s}}dt.
\endaligned\end{equation}
$\sigma_z(n)$ is the sum of the zth powers of the divisors of $n$. Namely,
 \begin{equation}\aligned\nonumber
\sigma_z(n):=\sum_{d| n} d^z.
\endaligned\end{equation}

$K$ is the second kind modified Bessel function, has the integral form \cite{Cohen2007}$($pages 113-117$)$
 \begin{equation}\aligned\nonumber
K_{\nu}(x)=\int_0^\infty e^{-x \cosh(t)}\cdot\cosh(\nu t)dt.
\endaligned\end{equation}
$K$ admits the asymptotic expansion at large $x$,
 \begin{equation}\aligned\nonumber
\lim_{x\rightarrow\infty}\frac{K_s(x)}{\sqrt{\frac{\pi}{2x}}e^{-x}}=1.
\endaligned\end{equation}

\end{theorem}
In view of \eqref{F31}, the second kind modified Bessel function plays an important role in our analysis.  The next proposition contains the explicit expansions of Bessel function in some special cases. For  a detailed analysis of Bessel function we refer to the monograph  \cite{Watson}.
\begin{proposition}[\cite{Watson}]\label{PropW}
When $s$ is half-integer $n+\frac{1}{2}$, $K_{n+\frac{1}{2}}$ admits an exact form
 \begin{equation}\aligned\nonumber
K_{n+\frac{1}{2}}(z)=\sqrt{\frac{\pi}{2z}}e^{-z}\sum_{k=0}^n \frac{(n+k)!}{k!(n-k)! (2z)^k}.
\endaligned\end{equation}

\end{proposition}

A consequence of the Chowla-Selberg formula in Theorem \ref{ThCS} yields the following.
\begin{lemma}\label{LemmaCS} We have the following expansion of $\Big(\zeta(6,z)-b\zeta(3,z)\Big)$
 \begin{equation}\aligned\nonumber
\Big(
\zeta(6,z)-b\zeta(3,z)
\Big)
=\Big(
2\xi(12)y^6+2\sqrt{\pi}\frac{\Gamma(\frac{11}{2})}{\Gamma(6)}\xi(11)y^{-5}
-b\big(2\xi(6)y^3+2\sqrt{\pi}\frac{\Gamma(\frac{5}{2})}{\Gamma(3)}\xi(5)y^{-2}
\big)
\Big)\\
+\sum_{n=1}^\infty\Big(
n^{\frac{11}{2}}\sigma_{-11}(n)\frac{8\pi^6}{\Gamma(6)}\sqrt{y}K_{\frac{11}{2}}(2\pi ny)
-bn^{\frac{5}{2}}\sigma_{-5}(n)\frac{8\pi^3}{\Gamma(3)}\sqrt{y}K_{\frac{5}{2}}(2\pi ny)
\Big)\cdot\cos(2\pi n x).
\endaligned\end{equation}

As a consequence, we have the expansion for derivatives
 \begin{equation}\aligned\nonumber
\frac{\partial}{\partial x}\Big(
\zeta(6,z)-b\zeta(3,z)
\Big)
=&2\pi\sin(2\pi x)
\sum_{n=1}^\infty\Big(
bn^{\frac{7}{2}}\sigma_{-5}(n)\frac{8\pi^3}{\Gamma(3)}\sqrt{y}K_{\frac{5}{2}}(2\pi ny)\\
&\;\;-
n^{\frac{13}{2}}\sigma_{-11}(n)\frac{8\pi^6}{\Gamma(6)}\sqrt{y}K_{\frac{11}{2}}(2\pi ny)
\Big)\cdot\frac{\sin(2\pi n x)}{\sin(2\pi  x)}.
\endaligned\end{equation}
\end{lemma}
For $K_{\frac{5}{2}}(2\pi ny)$ and $K_{\frac{11}{2}}(2\pi ny)$, from Proposition \ref{PropW}, direct computations give
\begin{lemma}\label{LemmaL}
\begin{equation}\aligned\nonumber
&\frac{K_{\frac{11}{2}}(2\pi ny)}{K_{\frac{11}{2}}(2\pi y)}\leq e^{-2\pi(n-1)y}\;\;\hbox{for}\;\;y>0,\\
&\frac{K_{\frac{11}{2}}(2\pi y)}{K_{\frac{5}{2}}(2\pi y)}\leq5.45\;\;\hbox{for}\;\;y\geq1,\\
&\frac{K_{\frac{11}{2}}(2\pi y)}{K_{\frac{5}{2}}(2\pi y)}>1\;\;\hbox{for}\;\;y>0.
\endaligned\end{equation}
\end{lemma}
There is another property of $\sqrt{y}K_{n+\frac{1}{2}}(2\pi ny),n\in\mathbb{Z}^+$ we shall use later. This follows from an observation in  Proposition \ref{PropW}. Namely,

\begin{lemma}\label{LemmaKKK}$\sqrt{y}K_{n+\frac{1}{2}}(2\pi ny),n\in\mathbb{Z}^+$ is completely monotone. i.e.,
\begin{equation}\aligned\nonumber
(-1)^j\frac{d^j}{dy^j}\{\sqrt{y}K_{n+\frac{1}{2}}(2\pi ny)\}>0,\;\;\hbox{for}\;\;n\in\mathbb{Z}^+, j\in\mathbb{Z}^0\;\;\hbox{and}\;\;y>0.
\endaligned\end{equation}
\end{lemma}
\begin{proof} Note that by Proposition \ref{PropW}, $\sqrt{y}K_{n+\frac{1}{2}}(2\pi ny),n\in\mathbb{Z}^+$
is a finite, positive and linear combination of terms of $e^{-ay}y^{-b}$, where $a,b>0$. On the other hand, it is easy to see that
$e^{-ay}$($a>0$) and $y^{-b}$($b>0$) both are completely monotone. Thus their product and finite, positive, and linear combination are still completely monotone (see e.g. \cite{Bet2016}).

\end{proof}

We shall state a comparison principle in estimates.
Namely,
\begin{lemma}[An comparison on $b$]\label{Lemma4.3} If there exists $b=\overline{b}>0$ such that
\begin{equation}\aligned\label{B1bb}
\frac{\partial}{\partial x}\Big(\zeta(6,z)-b\zeta(3,z)\Big)
\geq0,\;\;\hbox{for}\;\; z\in R,
\endaligned\end{equation}
where $R$ is any subset of $\overline{\mathcal{D}_{\mathcal{G}}}$.
Then for all $b\geq\overline{b}$, \eqref{B1bb} holds.
\end{lemma}

By Lemma \ref{Lemma4.3}, in proving Proposition \ref{Lemma81}, one only needs to consider a particular point of the parameter $b$.
The proof of Lemma \ref{Lemma4.3} is based on
\begin{equation}\aligned\nonumber
\frac{\partial}{\partial b}\Big(\frac{\partial}{\partial x}\Big(\zeta(6,z)-b\zeta(3,z)\Big)\Big)
=-\frac{\partial}{\partial x}\zeta(3,z)
\endaligned\end{equation}
and the following lemma by \cite{Ran1953}
\begin{lemma}
\begin{equation}\aligned\nonumber
\frac{\partial}{\partial x}\zeta(3,z)
\geq0,\;\;\hbox{for}\;\; z\in\overline{\mathcal{D}_{\mathcal{G}}}.
\endaligned\end{equation}
\end{lemma}
By Lemma \ref{Lemma4.3}, to prove Proposition \ref{Lemma81}, it suffices a particular case, namely
\begin{lemma}[{\bf x-monotonicity on a particular value}]\label{Lemma3.5} \begin{itemize}
                                                            \item [(1)] For $b=3.062$,
\begin{equation}\aligned\nonumber
\frac{\partial}{\partial x}\Big(\zeta(6,z)-3.062\zeta(3,z)\Big)
\geq0,\;\;\hbox{for}\;\; z\in\overline{\mathcal{D}_{\mathcal{G}}}.
\endaligned\end{equation}
                                                            \item [(2)] For $b=3$,
\begin{equation}\aligned\nonumber
\frac{\partial}{\partial x}\Big(\zeta(6,z)-3\zeta(3,z)\Big)
\geq0,\;\;\hbox{for}\;\; z\in\overline{\mathcal{D}_{\mathcal{G}}}\cap\{y\geq1\}.
\endaligned\end{equation}
                                                          \end{itemize}
\end{lemma}
To prove Lemma \ref{Lemma3.5}, we only prove the case $(1)$, since the proof of case $(2)$ is similar to a subcase of case $(1)$ as will see later. For convenience, we denote that
 \begin{equation}\aligned\label{HB100}
\mathcal{P}_n(y,b):&=\Big(
b n^{\frac{7}{2}}\sigma_{-5}(n)\frac{8\pi^3}{\Gamma(3)}\sqrt{y}K_{\frac{5}{2}}(2\pi ny)
-
n^{\frac{13}{2}}\sigma_{-11}(n)\frac{8\pi^6}{\Gamma(6)}\sqrt{y}K_{\frac{11}{2}}(2\pi ny)
\Big)\\
\mathcal{P}_n(y):&=\mathcal{P}_n(y,3.062).
\endaligned\end{equation}
As a consequence, one can rewrite $\frac{\partial}{\partial x}\Big(
\zeta(6,z)-b\zeta(3,z)
\Big)$ by
\begin{equation}\aligned\nonumber
\frac{\partial}{\partial x}\Big(
\zeta(6,z)-b\zeta(3,z)
\Big)
=&2\pi\sin(2\pi x)\cdot
\sum_{n=1}^\infty\mathcal{P}_n(y)
\cdot\frac{\sin(2\pi n x)}{\sin(2\pi  x)}.
\endaligned\end{equation}
Then case $(1)$ of Lemma \ref{Lemma3.5} equivalents to proving
\begin{lemma}\label{Lemma4.4}
\begin{equation}\aligned\nonumber
\sum_{n=1}^\infty\mathcal{P}_n(y)
\cdot\frac{\sin(2\pi n x)}{\sin(2\pi  x)}\geq0,\;\;\hbox{for}\;\; z=x+iy\in \overline{\mathcal{D}_{\mathcal{G}}}.
\endaligned\end{equation}
Here $\mathcal{P}_n(y)$ is defined in \eqref{HB100}.
\end{lemma}

The strategy of the proof of Lemma \ref{Lemma4.4} is as follows: by Chowla-Selberg formula we decompose $\frac{\partial}{\partial x}\Big(
\zeta(6,z)-b\zeta(3,z)
\Big)$ into two parts: $n\leq 3$ and $ n\geq 4$. We show that $n\leq 3$ parts dominate the rest.

To prove Lemma \ref{Lemma4.4}, one first has the following preliminary estimates
\begin{lemma}\label{Lemma3.6} For $y\geq\frac{\sqrt3}{2}$,
 \begin{equation}\aligned\nonumber
\mathcal{P}_n(y)&<0,\;\;\hbox{for}\;\;n\geq2.
\endaligned\end{equation}
For $\mathcal{P}_1(y)$, there exists $y_{\mathcal{P}_1}\cong0.9434111\cdots$ such that
 \begin{equation}\aligned\nonumber
\mathcal{P}_1(y)&<0,\;\;\hbox{for}\;\;y\in[\frac{\sqrt3}{2},y_{\mathcal{P}_1}).\\
\mathcal{P}_1(y)&>0,\;\;\hbox{for}\;\;y\in(y_{\mathcal{P}_1},\infty).
\endaligned\end{equation}
Here $\mathcal{P}_n(y)$ is defined in \eqref{HB100}.
\end{lemma}
\begin{proof}The fact that $\mathcal{P}_n(y)<0$ for $n\geq2$ and $y\geq\frac{\sqrt3}{2}$ follows directly from Lemma \ref{LemmaL}. The estimate on $\mathcal{P}_1(y)$ is by direct calculation.

\end{proof}

We shall divide the proof of Lemma \ref{Lemma4.4} into two cases, namely, {\bf case a:} $y\in[\frac{\sqrt3}{2},1]$, {\bf case b}: $y\geq1$.
Note that $y\in[\frac{\sqrt3}{2},1]$ and $z=x+iy\in \overline{\mathcal{D}_{\mathcal{G}}}$ imply that $x\in[\sqrt{1-y^2},\frac{1}{2}]$.

We shall analyze the main order terms $\sum_{n=1}^3\mathcal{P}_n(y)
\cdot\frac{\sin(2\pi nx)}{\sin(2\pi x)}$ and the error terms
$\sum_{n=4}^\infty\mathcal{P}_n(y)
\cdot\frac{\sin(2\pi nx)}{\sin(2\pi x)}$ respectively in the following.
For the main order terms $\sum_{n=1}^3\mathcal{P}_n(y)
\cdot\frac{\sin(2\pi nx)}{\sin(2\pi x)}$, direct calculation shows that
\begin{equation}\aligned\label{LemmaHM67}
\sum_{n=1}^3\mathcal{P}_n(y)\cdot\frac{\sin(2\pi nx)}{\sin(2\pi x)}
=\mathcal{P}_1(y)+\mathcal{P}_3(y)+
2\mathcal{P}_2(y)\cos(2\pi x)
+2\mathcal{P}_3(y)\cos(4\pi x),
\endaligned\end{equation}
\begin{equation}\aligned\label{HL100}
\frac{\partial}{\partial x}\Big(\sum_{n=1}^3\mathcal{P}_n(y)\cdot\frac{\sin(2\pi nx)}{\sin(2\pi x)}\Big)
=-2\sin(2\pi x)\cdot\big(\mathcal{P}_2(y)+2\mathcal{P}_3(y)\cos(2\pi x)\big).
\endaligned\end{equation}
Here $\mathcal{P}_n(y)$ is defined in \eqref{HB100}.
We shall prove that

\begin{lemma}\label{LemmaHM68}For $z=x+iy\in \overline{\mathcal{D}_{\mathcal{G}}}$, it holds that
\begin{equation}\aligned\nonumber
\frac{\partial}{\partial x}\Big(\sum_{n=1}^3\mathcal{P}_n(y)\cdot\frac{\sin(2\pi nx)}{\sin(2\pi x)}\Big)
\geq0.
\endaligned\end{equation}

\end{lemma}

By \eqref{HL100}, the proof of Lemma \ref{LemmaHM68} follows from the following

\begin{lemma}\label{LemmaP23b} For $y\geq\frac{\sqrt3}{2}$,
\begin{equation}\aligned\nonumber
2\mathcal{P}_3(y)-\mathcal{P}_2(y)\geq\frac{1}{10}\cdot2^{\frac{13}{2}}\sigma_{-11}(2)\frac{8\pi^6}{\Gamma(6)}\cdot K_{\frac{11}{2}}(4\pi y)>0.
\endaligned\end{equation}
Here $\mathcal{P}_n(y)$ is defined in \eqref{HB100}.
\end{lemma}

\begin{proof}
By \eqref{HB100}, one has the explicit expression,
\begin{equation}\aligned\nonumber
2\mathcal{P}_3(y)-\mathcal{P}_2(y)
&=
2^{\frac{13}{2}}\sigma_{-11}(2)\frac{8\pi^6}{\Gamma(6)}K_{\frac{11}{2}}(4\pi y)+
2\cdot3.062\cdot3^{\frac{7}{2}}\sigma_{-5}(3)\frac{8\pi^3}{\Gamma(3)}K_{\frac{5}{2}}(6\pi y)\\
&\;\;\;\;-3.062\cdot2^{\frac{7}{2}}\sigma_{-11}(2)\frac{8\pi^3}{\Gamma(3)}K_{\frac{5}{2}}(4\pi y)-
2\cdot3^{\frac{13}{2}}\sigma_{-11}(3)\frac{8\pi^6}{\Gamma(6)}K_{\frac{11}{2}}(6\pi y).
\endaligned\end{equation}
Dropping the second positive term, we find that
\begin{equation}\aligned\label{Lina}
2\mathcal{P}_3(y)-\mathcal{P}_2(y)
&\geq2^{\frac{13}{2}}\sigma_{-11}(2)\frac{8\pi^6}{\Gamma(6)}K_{\frac{11}{2}}(4\pi y)\cdot
\Big(
1-\frac{3.062\Gamma(6)}{8\pi^3\Gamma(3)}\frac{K_{\frac{5}{2}}(4\pi y)}{K_{\frac{11}{2}}(4\pi y)}-2(\frac{3}{2})^{\frac{13}{2}}\frac{\sigma_{-11}(3)}{\sigma_{-11}(2)}\frac{K_{\frac{11}{2}}(6\pi y)}{K_{\frac{11}{2}}(4\pi y)}
\Big)\\
&\geq2^{\frac{13}{2}}\sigma_{-11}(2)\frac{8\pi^6}{\Gamma(6)}K_{\frac{11}{2}}(4\pi y)\cdot
\Big(
1-\frac{3.062\Gamma(6)}{8\pi^3\Gamma(3)}-2(\frac{3}{2})^{\frac{13}{2}}\frac{\sigma_{-11}(3)}{\sigma_{-11}(2)}e^{-2\pi y}
\Big),
\endaligned\end{equation}
where we have used the fact that
\begin{equation}\aligned\nonumber
\frac{K_{\frac{5}{2}}(4\pi y)}{K_{\frac{11}{2}}(4\pi y)}\leq1,\;\;\hbox{and}\;\;
\frac{K_{\frac{11}{2}}(6\pi y)}{K_{\frac{11}{2}}(4\pi y)}
\leq e^{-2\pi y}
\endaligned\end{equation}
by Lemma \ref{LemmaL}. The proof is complete by \eqref{Lina} and the following
\begin{equation}\aligned\nonumber
\frac{3.062\Gamma(6)}{8\pi^3\Gamma(3)}\leq0.740656\cdots,\;\;\hbox{and}\;\;
2(\frac{3}{2})^{\frac{13}{2}}\frac{\sigma_{-11}(3)}{\sigma_{-11}(2)}e^{-\sqrt3\pi }\leq0.120907\cdots.
\endaligned\end{equation}

\end{proof}

By Lemma \ref{LemmaHM68} and \eqref{LemmaHM67}, one has
\begin{lemma}\label{LemmaHM610}For $z=x+iy\in \overline{\mathcal{D}_{\mathcal{G}}}$, it holds that
\begin{equation}\aligned\nonumber
\sum_{n=1}^3\mathcal{P}_n(y)\cdot\frac{\sin(2\pi nx)}{\sin(2\pi x)}
\geq\mathcal{P}_1(y)+\mathcal{P}_3(y)+2\mathcal{P}_2(y)\cos(2\pi\sqrt{1-y^2})+2\mathcal{P}_3(y)\cos(4\pi\sqrt{1-y^2}).
\endaligned\end{equation}
Here $\mathcal{P}_n(y)$ is defined in \eqref{HB100}.
\end{lemma}
For controlling the error terms
$\sum_{n=4}^\infty\mathcal{P}_n(y)
\cdot\frac{\sin(2\pi nx)}{\sin(2\pi x)}$, we observe that
\begin{lemma}\label{LemmaHB610} For $x\in\R, n\in\mathbb{Z}^+$, it holds that
\begin{equation}\aligned\label{LemmaSin}
|\frac{\sin(2n\pi x)}{\sin(2\pi x)}|\leq n.
\endaligned\end{equation}
As a consequence,
 for $y\geq\frac{\sqrt3}{2}$,
\begin{equation}\aligned\nonumber
|\sum_{n=4}^\infty\mathcal{P}_n(y)\cdot\frac{\sin(2\pi nx)}{\sin(2\pi x)}|
\leq
\epsilon_1(y),
\endaligned\end{equation}
where
\begin{equation}\aligned\label{HB101}
\epsilon_1(y):=\sum_{n=4}^\infty n^{\frac{15}{2}}\sigma_{-11}(n)\frac{8\pi^6}{\Gamma(6)}\cdot\sqrt yK_{\frac{11}{2}}(2\pi ny).
\endaligned\end{equation}
Here $\mathcal{P}_n(y)$ and $\epsilon_1(y)$ are defined in \eqref{HB100} and \eqref{HB101} respectively.
\end{lemma}


By Lemmas \ref{LemmaHM610} and \ref{LemmaHB610}, one has

\begin{lemma}\label{Lemma4.5} For $z\in\overline{\mathcal{D}_{\mathcal{G}}}\cap\{y\in[\frac{\sqrt3}{2},1]\}$,
\begin{equation}\aligned\nonumber
\sum_{n=1}^\infty\mathcal{P}_n(y)
\cdot\frac{\sin(2\pi n x)}{\sin(2\pi  x)}
\geq\mathcal{P}_1(y)+\mathcal{P}_3(y)+2\mathcal{P}_2(y)\cos(2\pi\sqrt{1-y^2})+2\mathcal{P}_3(y)\cos(4\pi\sqrt{1-y^2})-\epsilon_1(y).
\endaligned\end{equation}
\end{lemma}
By Lemma \ref{Lemma4.5}, the proof of {\bf case a:} $y\in[\frac{\sqrt3}{2},1]$ of Lemma \ref{Lemma4.4} is proved by the following
\begin{lemma}\label{Lemma319} For $y\in[\frac{\sqrt3}{2},1]$,
\begin{equation}\aligned\nonumber
\mathcal{P}_1(y)+\mathcal{P}_3(y)+2\mathcal{P}_2(y)\cos(2\pi\sqrt{1-y^2})+2\mathcal{P}_3(y)\cos(4\pi\sqrt{1-y^2})
&\geq2.655\cdot 10^{-5};\\
\epsilon_1(y)&\leq 5.76\cdot 10^{-6}.
\endaligned\end{equation}
\end{lemma}
Lemma \ref{Lemma319} is proved by a direct computation.

It remains to prove {\bf case b:} $y\in[1,\infty)$ of Lemma \ref{Lemma4.4}, which follows from
\begin{lemma}\label{Lemma4.6} For $z\in\overline{\mathcal{D}_{\mathcal{G}}}\cap\{y\geq1\}$,
\begin{equation}\aligned\nonumber
\sum_{n=1}^\infty\mathcal{P}_n(y)
\cdot\frac{\sin(2\pi n x)}{\sin(2\pi  x)}
\geq\frac{1}{2}\mathcal{P}_1(y)>0.
\endaligned\end{equation}
Note that $\mathcal{P}_1(y)>0$ for $y\geq1$ by Lemma \ref{Lemma3.6}.
\end{lemma}
\begin{proof}
 We estimate by selecting the leading order term $\mathcal{P}_1(y)$ and by Lemmas \ref{LemmaSin} and \ref{Lemma3.6},
\begin{equation}\aligned\nonumber
\sum_{n=1}^\infty\mathcal{P}_n(y)
\cdot\frac{\sin(2\pi n x)}{\sin(2\pi  x)}
&=\mathcal{P}_1(y)\cdot
\Big(
1+\sum_{n=2}^\infty \frac{\mathcal{P}_n(y)}{\mathcal{P}_1(y)}\cdot \frac{\sin(2\pi n x)}{\sin(2\pi  x)}
\Big)\\
&\geq
\mathcal{P}_1(y)\cdot
\Big(
1-\sum_{n=2}^\infty n\cdot |\frac{\mathcal{P}_n(y)}{\mathcal{P}_1(y)}|
\Big)\\
&\geq
\mathcal{P}_1(y)\cdot
\Big(
1-\frac{\frac{8\pi^6}{\Gamma(6)}}{3.062\frac{8\pi^3}{\Gamma(3)}}\sum_{n=2}^\infty
n^{\frac{15}{2}}\sigma_{-11}(n)\frac{K_{\frac{11}{2}}(2\pi ny)}{K_{\frac{5}{2}}(2\pi y)}
\Big)\\
&\geq
\mathcal{P}_1(y)\cdot
\Big(
1-\frac{5.45\frac{8\pi^6}{\Gamma(6)}}{3.062\frac{8\pi^3}{\Gamma(3)}}\cdot\sum_{n=2}^\infty
n^{\frac{15}{2}}\sigma_{-11}(n)\cdot e^{-2\pi(n-1)y}
\Big)\\
&\geq
\frac{1}{2}\mathcal{P}_1(y).
\endaligned\end{equation}
Here we used that fact that
\begin{equation}\aligned\nonumber
\frac{5.45\frac{8\pi^6}{\Gamma(6)}}{3.062\frac{8\pi^3}{\Gamma(3)}}\cdot\sum_{n=2}^\infty
n^{\frac{15}{2}}\sigma_{-11}(n)\cdot e^{-2\pi(n-1)y}
\leq0.3234\cdots<\frac{1}{2},\;\;\hbox{for}\;\;y\geq1
\endaligned\end{equation}
and that
for $y\geq1$
\begin{equation}\aligned\nonumber
\frac{K_{\frac{11}{2}}(2\pi ny)}{K_{\frac{5}{2}}(2\pi y)}
\leq\frac{K_{\frac{11}{2}}(2\pi ny)}{K_{\frac{11}{2}}(2\pi y)}
\cdot\frac{K_{\frac{11}{2}}(2\pi y)}{K_{\frac{5}{2}}(2\pi y)}\leq5.45e^{-2\pi(n-1)y}
\endaligned\end{equation}
by Lemma \ref{LemmaL}.
\end{proof}

\section{\bf the cases $b\leq3.06$: y-convexity and y-monotonicity of x-derivative: }
Recall that
\begin{equation}\aligned\nonumber
\Gamma_b=\{
z\in\mathbb{H}: |z|=1,\; \Im(z)\in[\frac{\sqrt3}{2},1]
\}.
\endaligned\end{equation}

In this section, we aim to prove that
\begin{theorem}\label{Th102} For $b\leq3.06$,
\begin{equation}\aligned\nonumber
\min_{z\in\overline{\mathcal{D}_{\mathcal{G}}}}\Big(\zeta(6,z)-b\zeta(3,z)\Big)
=\min_{z\in \Gamma_b}\Big(\zeta(6,z)-b\zeta(3,z)\Big).
\endaligned\end{equation}

\end{theorem}
The proof of Theorem \ref{Th102} is quite involved and is based on careful investigation of two second order estimates as stated in the two following propositions.
\begin{proposition}[{\bf y-convexity}]\label{Lemma101} For $b\leq3.06$,
\begin{equation}\aligned\nonumber
\frac{\partial^2}{\partial y^2}\Big(\zeta(6,z)-b\zeta(3,z)\Big)
>0,\;\;\hbox{for}\;\; z\in\overline{\mathcal{D}_{\mathcal{G}}}.
\endaligned\end{equation}
\end{proposition}
\begin{proposition}[{\bf y-monotonicity of x-derivative}]\label{Lemma102} For $b\leq3.07$,
\begin{equation}\aligned\nonumber
\frac{\partial^2}{\partial y\partial x}\Big(\zeta(6,z)-b\zeta(3,z)\Big)
>0,\;\;\hbox{for}\;\; z=x+iy\in \{x\in[\frac{\sqrt3}{2}],y\in[0,\frac{1}{2}]\}.
\endaligned\end{equation}
\end{proposition}
 Propositions \ref{Lemma101} and \ref{Lemma101} are indeed motivated by an alternative monotonicity property for the functional
$\Big(
\zeta(6,z)-b\zeta(3,z)
\Big)$ on the $\frac{1}{4}-$arc $\Gamma_b$ as follows.

\begin{proposition}\label{Prop101}We have the following alternative monotonicity on $\Gamma_b$:
for any $z\in \Gamma_b$ and any $b\in\R$,

  \item
\begin{itemize}
  \item either $\frac{\partial}{\partial x}\Big(
\zeta(6,z)-b\zeta(3,z)
\Big)\geq0$,
  \item or $\frac{\partial}{\partial y}\Big(
\zeta(6,z)-b\zeta(3,z)
\Big)\geq0$.
\end{itemize}

\end{proposition}

\begin{proof}
Denote that
\begin{equation}\aligned\nonumber
\mathcal{W}_b(z):=\Big(\zeta(6,z)-b\zeta(3,z)\Big).
\endaligned\end{equation}

Since $z\mapsto -\frac{1}{z}, z\mapsto -\overline{z}\in\mathcal{G}$, by Lemma \ref{G111},
\begin{equation}\aligned\label{F1}
\mathcal{W}_b(z)=\mathcal{W}_b(\frac{1}{\overline{z}}).
\endaligned\end{equation}
Let $z=r(\cos(\theta)+i\sin(\theta))$ be the polar coordinate form, \eqref{F1} reads as
\begin{equation}\aligned\label{F2}
\mathcal{W}_b(r(\cos(\theta)+i\sin(\theta)))=\mathcal{W}_b(\frac{1}{r}(\cos(\theta)+i\sin(\theta))).
\endaligned\end{equation}
Taking derivative with respect to $r$ on \eqref{F2}, one has
\begin{equation}\aligned\label{F3}
\frac{\partial}{\partial r}\mathcal{W}_b(r(\cos(\theta)+i\sin(\theta)))\mid_{r=1}=0.
\endaligned\end{equation}
Which implies that
\begin{equation}\aligned\label{F4}
\frac{\partial}{\partial x}\mathcal{W}_b(\cos(\theta)+i\sin(\theta))\cdot\cos(\theta)
=-\frac{\partial}{\partial y}\mathcal{W}_b(\cos(\theta)+i\sin(\theta))\cdot\sin(\theta).
\endaligned\end{equation}
Note for $z\in\Gamma_b$, $z=\cos(\theta)+i\sin(\theta)$ with $\theta\in[\frac{\pi}{3},\frac{\pi}{2}]$.
Therefore the Proposition follows by \eqref{F4}.
\end{proof}

We postpone the proof of Propositions \ref{Lemma101} and \ref{Lemma102} to late of this section and give the proof of
Theorem \ref{Th102} based on Propositions \ref{Lemma101} and \ref{Lemma102}.

\begin{proof}[Proof of Theorem \ref{Th102}]
We divide the fundamental region $\overline{\mathcal{D}_{\mathcal{G}}}$ into two different sub-regions. We refer the readers to see Picture \ref{PFFF}(right one) when reading this proof. In each subregion, we use different
strategies.

\noindent

{\bf Case A: $\overline{\mathcal{D}_{\mathcal{G}}}\cap\{y\geq1\}.$}
Taking $\theta=\frac{\pi}{2}$ in \eqref{F4}, one has
$$\frac{\partial}{\partial y}\Big(\zeta(6,z)-b\zeta(3,z)\Big)\mid_{z=i}=0\;\;\hbox{for any}\;\; b\in\R.$$
Therefore, by Proposition \ref{Lemma102}, for $b\leq3.06$,
\begin{equation}\aligned\label{H111}
\frac{\partial}{\partial y}\Big(\zeta(6,z)-b\zeta(3,z)\Big)
>0,\;\;\hbox{for}\;\; z=x+iy\in \{y=1,x\in(0,\frac{1}{2}]\}.
\endaligned\end{equation}
It follows by Proposition \ref{Lemma101} and \eqref{H111} that, for $b\leq3.06$
\begin{equation}\aligned\nonumber
\frac{\partial}{\partial y}\Big(\zeta(6,z)-b\zeta(3,z)\Big)
\geq0,\;\;\hbox{for}\;\; z\in \overline{\mathcal{D}_{\mathcal{G}}}\cap\{y\geq1\}.
\endaligned\end{equation}
It yields that, for $b\leq3.06$
\begin{equation}\aligned\nonumber
\min_{z\in\overline{\mathcal{D}_{\mathcal{G}}}\cap\{y\geq1\}}\Big(\zeta(6,z)-b\zeta(3,z)\Big)
=\min_{z\in\{y=1,x\in[0,\frac{1}{2}]\}}\Big(\zeta(6,z)-b\zeta(3,z)\Big).
\endaligned\end{equation}
Therefore, we then have, for $b\leq3.06$
\begin{equation}\aligned\label{JJJ}
\min_{z\in\overline{\mathcal{D}_{\mathcal{G}}}}\Big(\zeta(6,z)-b\zeta(3,z)\Big)
=\min_{z\in\overline{\mathcal{D}_{\mathcal{G}}}\cap\{y\in[\frac{\sqrt3}{2},1]\}}\Big(\zeta(6,z)-b\zeta(3,z)\Big).
\endaligned\end{equation}

{\bf Case B: $\overline{\mathcal{D}_{\mathcal{G}}}\cap\{y\in[\frac{\sqrt3}{2},1]\}.$} We aim to prove that
\begin{equation}\aligned\nonumber
\hbox{For}\;b\leq3.06,\;\;\min_{z\in\overline{\mathcal{D}_{\mathcal{G}}}\cap\{y\in[\frac{\sqrt3}{2},1]\}}\Big(\zeta(6,z)-b\zeta(3,z)\Big)
=\min_{z\in\Gamma_b}\Big(\zeta(6,z)-b\zeta(3,z)\Big).
\endaligned\end{equation}

To this end and for convenience, we denote two regions
\begin{equation}
\Gamma_u:=\{y=1,x\in(0,\frac{1}{2}]\},\
\Gamma_r:=\{x=\frac{1}{2},y\in(\frac{\sqrt3}{2},1]\}.
\end{equation}
For $\;b\leq3.06$, we assume that $
\min_{z\in\overline{\mathcal{D}_{\mathcal{G}}}\cap\{y\in[\frac{\sqrt3}{2},1]\}} \Big(\zeta(6,z)-b\zeta(3,z)\Big)
\;\;\hbox{is achieved at}\;\; z_{b,0}.$

Note that, $\Gamma_b,\Gamma_u$ and $\Gamma_r$ consist the boundaries of the small finite region $\overline{\mathcal{D}_{\mathcal{G}}}\cap\{y\in[\frac{\sqrt3}{2},1]\}$(see Picture \ref{PFFF}).
We prove by excluding the possibility of the minimizers occur on the upper and right boundaries, i.e., $\Gamma_u,\Gamma_r$ and the interior points respectively. We consider three sub-cases

{\bf Subcase B$_1$: $z_{b,0}\notin\Gamma_u$ for $b\leq3.06$.} In fact, this is a direct consequence of \eqref{H111}.

{\bf Subcase B$_2$: $z_{b,0}\notin\Gamma_r$ for $b\leq3.06$.} It follows by
\begin{equation}\aligned\label{H222}
\hbox{for}\;b\leq3.06,\;\;\frac{\partial}{\partial y}\Big(\zeta(6,z)-b\zeta(3,z)\Big)
>0,\;\;\hbox{for}\;\; z\in \Gamma_r.
\endaligned\end{equation}

 Note that $\frac{\partial}{\partial y}\zeta(s,z)\mid_{z=\frac{1}{2}+i\frac{\sqrt3}{2}}=0$ for $s>0$\cite{Ran1953,Cas1959},
then $\frac{\partial}{\partial y}\Big(\zeta(6,z)-b\zeta(3,z)\Big)\mid_{z=\frac{1}{2}+i\frac{\sqrt3}{2}}=0$ for any $b\in\R$.
Therefore, by Proposition \ref{Lemma101}, \eqref{H222} follows.

{\bf Subcase B$_3$: $z_{b,0}\notin$ interior points of $\overline{\mathcal{D}_{\mathcal{G}}}\cap\{y\in[\frac{\sqrt3}{2},1]\}$ for $b\leq3.06$.}
 We use a contradiction argument. Suppose that $z_{b,0}$ is a interior point of $\overline{\mathcal{D}_{\mathcal{G}}}\cap\{y\in[\frac{\sqrt3}{2},1]\}$ for $b\leq3.06$. Then it holds that
\begin{equation}\aligned\label{H333}
\hbox{for}\;b\leq3.06,\;\;\frac{\partial}{\partial y}\Big(\zeta(6,z)-b\zeta(3,z)\Big)\mid_{z=z_{b,0}}=0,\;\;\hbox{and}\;\;
\frac{\partial}{\partial x}\Big(\zeta(6,z)-b\zeta(3,z)\Big)\mid_{z=z_{b,0}}=0.
\endaligned\end{equation}
Further denote that $z_{b,0}=x_0+iy_0$.
Consider point $x_0+i\sqrt{1-x_0^2}\in\Gamma_b$. By Proposition \ref{Prop101}, there holds one of the following two sub-sub-cases:

{\bf Sub-sub-case B$_{31}$: $\frac{\partial}{\partial y}\Big(
\zeta(6,z)-b\zeta(3,z)
\Big)\mid_{z=x_0+i\sqrt{1-x_0^2}}\geq0$.
} In this sub-sub-case, by Proposition \ref{Lemma101}, $\frac{\partial}{\partial y}\Big(
\zeta(6,z)-b\zeta(3,z)
\Big)\mid_{z=z_{b,0}}>0$. This contradicts to \eqref{H333}.

{\bf Sub-sub-case B$_{32}$: $\frac{\partial}{\partial x}\Big(
\zeta(6,z)-b\zeta(3,z)
\Big)\mid_{z=x_0+i\sqrt{1-x_0^2}}\geq0$.
} In this sub-sub-case, by Proposition \ref{Lemma102}, $\frac{\partial}{\partial x}\Big(
\zeta(6,z)-b\zeta(3,z)
\Big)\mid_{z=z_{b,0}}>0$. This still contradicts to \eqref{H333}.

Combining all cases above, we complete the proof of Theorem \ref{Th102}.

\end{proof}

In the rest of this section, we prove Propositions \ref{Lemma101} and \ref{Lemma102} separately.

\subsection{Proof of Proposition \ref{Lemma101}}

In this subsection, we shall prove the $y$-convexity stated in Proposition \ref{Lemma101}.

We first state  the following useful comparison principle.
\begin{lemma}[An comparison on $b$ of y-convexity]\label{LemmaA1} If there exists $b=\overline{b}>0$ such that
\begin{equation}\aligned\label{B1}
\frac{\partial^2}{\partial y^2}\Big(\zeta(6,z)-b\zeta(3,z)\Big)
>0,\;\;\hbox{for}\;\; z\in\overline{\mathcal{D}_{\mathcal{G}}}.
\endaligned\end{equation}
Then for all $b\leq\overline{b}$, \eqref{B1} holds.
\end{lemma}

By Lemma \ref{LemmaA1}, to prove Proposition \ref{Lemma101}, one only needs to consider a particular point of the parameter $b$.
The proof of Lemma \ref{LemmaA1} is based on
\begin{equation}\aligned\nonumber
\frac{\partial}{\partial b}\Big(\frac{\partial^2}{\partial y^2}\Big(\zeta(6,z)-b\zeta(3,z)\Big)\Big)
=-\frac{\partial^2}{\partial y^2}\zeta(3,z)
\endaligned\end{equation}
and the following lemma
\begin{lemma}[{\bf $y-$convexity of $\zeta(3,z)$}]\label{LemmaA2a}
\begin{equation}\aligned\nonumber
\frac{\partial^2}{\partial y^2}\zeta(3,z)\geq4
>0,\;\;\hbox{for}\;\; z\in\overline{\mathcal{D}_{\mathcal{G}}}.
\endaligned\end{equation}
\end{lemma}
Note that $z\in\overline{\mathcal{D}_{\mathcal{G}}}$ implies that $y\geq\frac{\sqrt3}{2}$.
In fact, we will show that
\begin{lemma}[{\bf $y-$convexity of $\zeta(3,z)$}]\label{LemmaA2b}
\begin{equation}\aligned\label{B1}
\frac{\partial^2}{\partial y^2}\zeta(3,z)\geq4
>0,\;\;\hbox{for}\;\; z\in \{x\in\R\}\cap \{y\in[\frac{\sqrt3}{2},\infty)\}.
\endaligned\end{equation}
\end{lemma}

\begin{proof}[Proof of Lemma \ref{LemmaA2b}] By Theorem \ref{ThCS},
\begin{equation}\aligned\nonumber
\frac{\partial ^2}{\partial y^2}
\zeta(3,z)
=&
\big(12\xi(6)y+12\sqrt{\pi}\frac{\Gamma(\frac{5}{2})}{\Gamma(3)}\xi(5)y^{-4}
\big)\\
&+\sum_{n=1}^\infty\Big(
n^{\frac{5}{2}}\sigma_{-5}(n)\frac{8\pi^3}{\Gamma(3)}\frac{\partial ^2}{\partial y^2}\{\sqrt{y}K_{\frac{5}{2}}(2\pi ny)\}
\Big)\cdot\cos(2\pi n x).
\endaligned\end{equation}

By the positiveness, comparisons and complete monotonicity of $\frac{\partial ^2}{\partial y^2}\{\sqrt{y}K_{\frac{5}{2}}(2\pi ny)\}$ in Lemmas \ref{LemmaAA17}, \ref{LemmaAA18} and \ref{LemmaKKK} respectively, we estimate that
\begin{equation}\aligned\label{LH1}
\frac{\partial ^2}{\partial y^2}
\zeta(3,z)
\geq&
\big(12\xi(6)y+12\sqrt{\pi}\frac{\Gamma(\frac{5}{2})}{\Gamma(3)}\xi(5)y^{-4}
\big)\\
&-\sum_{n=1}^\infty\Big(
n^{\frac{5}{2}}\sigma_{-5}(n)\frac{8\pi^3}{\Gamma(3)}\frac{\partial ^2}{\partial y^2}\{\sqrt{y}K_{\frac{5}{2}}(2\pi ny)\}
\Big)\\
\geq&
\big(12\xi(6)y+12\sqrt{\pi}\frac{\Gamma(\frac{5}{2})}{\Gamma(3)}\xi(5)y^{-4}
\big)\\
&-\sum_{n=1}^\infty\Big(
n^{4}\sigma_{-5}(n)\frac{8\pi^3}{\Gamma(3)}\cdot e^{-2\pi(n-1)y}\cdot\frac{\partial ^2}{\partial y^2}\{\sqrt{y}K_{\frac{5}{2}}(2\pi y)\}
\Big).
\endaligned\end{equation}
Note that $z\in\overline{\mathcal{D}_{\mathcal{G}}}$ implies that $y\in[\frac{\sqrt3}{2},\infty)$.
We then divide the proof into two cases, namely, {\bf case a: $y\in[\frac{\sqrt3}{2},1]$}, {\bf case b: $y\in[1,\infty)$}.

For {\bf case a: $y\in[\frac{\sqrt3}{2},1]$}, by \eqref{LH1},
\begin{equation}\aligned\nonumber
\frac{\partial ^2}{\partial y^2}
\zeta(3,z)
\geq&
\big(12\xi(6)+12\sqrt{\pi}\frac{\Gamma(\frac{5}{2})}{\Gamma(3)}\xi(5)
\big)-\frac{8\pi^3}{\Gamma(3)}\cdot\frac{\partial ^2}{\partial y^2}\{\sqrt{y}K_{\frac{5}{2}}(2\pi y)\}\mid_{y=\frac{\sqrt3}{2}}\cdot\sum_{n=1}^\infty\Big(
n^{4}\sigma_{-5}(n) e^{-\sqrt 3\pi(n-1)}
\Big)\\
\geq&26.8-22.6>4.
\endaligned\end{equation}
Here we have used the fact that the monotonicity of $\frac{\partial ^2}{\partial y^2}\{\sqrt{y}K_{\frac{5}{2}}(2\pi y)\}$ in Lemma \ref{LemmaKKK}, and the fact that
\begin{equation}\aligned\nonumber
\frac{8\pi^3}{\Gamma(3)}=124.025106\cdots\\
  \frac{\partial^2}{\partial y^2}\{\sqrt{y}K_{\frac{5}{2}}(2\pi y)\}\mid_{y=\frac{\sqrt3}{2}}=0.170011\cdots\\
  \sum_{n=1}^\infty\Big(
n^{4}\sigma_{-5}(n)\cdot e^{-\sqrt 3\pi(n-1)}\leq1.104342\cdots.
\endaligned\end{equation}

For {\bf case b: $y\in[1,\infty)$}, by \eqref{LH1},
\begin{equation}\aligned\nonumber
\frac{\partial ^2}{\partial y^2}
\zeta(3,z)
\geq&
5\sqrt[5]{(3\xi(6))^4\cdot12\sqrt{\pi}\frac{\Gamma(\frac{5}{2})}{\Gamma(6)}\xi(5)}
-\frac{8\pi^3}{\Gamma(3)}\cdot\frac{\partial ^2}{\partial y^2}\{\sqrt{y}K_{\frac{5}{2}}(2\pi y)\}\mid_{y=1}\cdot\sum_{n=1}^\infty\Big(
n^{4}\sigma_{-5}(n) e^{-2\pi(n-1)}
\Big)\\
\geq&20.8-8.5>12.
\endaligned\end{equation}
Here we have used that the mean value inequality, namely, $ay+b y^4\geq5\cdot\sqrt[5]{(\frac{a}{4})^4\cdot b}$ for $a,b,y>0$ and the monotonicity of $\frac{\partial ^2}{\partial y^2}\{\sqrt{y}K_{\frac{5}{2}}(2\pi y)\}$ in Lemma \ref{LemmaKKK}, and the following estimates
\begin{equation}\aligned\nonumber
\frac{8\pi^3}{\Gamma(3)}=124.025106\cdots\\
  \frac{\partial^2}{\partial y^2}\{\sqrt{y}K_{\frac{5}{2}}(2\pi y)\}\mid_{y=1}=0.065966\cdots\\
  \sum_{n=1}^\infty\Big(
n^{4}\sigma_{-5}(n)\cdot e^{-\sqrt 2\pi(n-1)}\leq1.062356\cdots.
\endaligned\end{equation}

\end{proof}

By Lemma \ref{LemmaA1}, to prove Proposition \ref{Lemma101}, it suffices to prove that
\begin{lemma}[{\bf y-convexity on a particular value}]\label{LemmaA8} For $b=3.06$,
\begin{equation}\aligned\nonumber
\frac{\partial^2}{\partial y^2}\Big(\zeta(6,z)-3.06\zeta(3,z)\Big)
>0,\;\;\hbox{for}\;\; z\in\overline{\mathcal{D}_{\mathcal{G}}}.
\endaligned\end{equation}
\end{lemma}

To this end, we note that by Lemma \ref{LemmaCS}, one has

\begin{lemma}\label{Lemmayy} We have the expansion of $\frac{\partial ^2}{\partial y^2}\Big(\zeta(6,z)-b\zeta(3,z)\Big)$
 \begin{equation}\aligned\nonumber
\frac{\partial ^2}{\partial y^2}\Big(
\zeta(6,z)-b\zeta(3,z)
\Big)
=\Big(
60\xi(12)y^4+60\sqrt{\pi}\frac{\Gamma(\frac{11}{2})}{\Gamma(6)}\xi(11)y^{-7}
-b\big(12\xi(6)y+12\sqrt{\pi}\frac{\Gamma(\frac{5}{2})}{\Gamma(3)}\xi(5)y^{-4}
\big)
\Big)\\
+\sum_{n=1}^\infty\Big(
n^{\frac{11}{2}}\sigma_{-11}(n)\frac{8\pi^6}{\Gamma(6)}\frac{\partial ^2}{\partial y^2}\{\sqrt{y}K_{\frac{11}{2}}(2\pi ny)\}
-bn^{\frac{5}{2}}\sigma_{-5}(n)\frac{8\pi^3}{\Gamma(3)}\frac{\partial ^2}{\partial y^2}\{\sqrt{y}K_{\frac{5}{2}}(2\pi ny)\}
\Big)\cdot\cos(2\pi n x).
\endaligned\end{equation}
\end{lemma}
To analyze $\frac{\partial ^2}{\partial y^2}\{\sqrt{y}K_{\frac{5}{2}}(2\pi ny)\}$ and $\frac{\partial ^2}{\partial y^2}\{\sqrt{y}K_{\frac{11}{2}}(2\pi ny)\}$, we start with some preliminary estimates
\begin{lemma}\label{LemmaBtwo}There holds
\begin{equation}\aligned\nonumber
\frac{\partial^2}{\partial y^2}\{
\sqrt{y}K_{\frac{5}{2}}(2n\pi y)\}
=&2n^{\frac{3}{2}}\pi^2e^{-2\pi ny}
\Big(
1+\frac{3}{2n\pi}\cdot\frac{1}{y}+\frac{9}{4n^2\pi^2}\cdot\frac{1}{y^2}
+\frac{9}{4n^3\pi^3}\cdot\frac{1}{y^3}+\frac{9}{8n^4\pi^4}\cdot\frac{1}{y^4}
\Big)\\
\frac{\partial^2}{\partial y^2}\{
\sqrt{y}K_{\frac{11}{2}}(2n\pi y)\}
=&2n^{\frac{3}{2}}\pi^2e^{-2\pi ny}
\Big(
1+\frac{15}{n\pi}\cdot\frac{1}{y}+\frac{135}{4n^2\pi^2}\cdot\frac{1}{y^2}
+\frac{435}{4n^3\pi^3}\cdot\frac{1}{y^3}+\frac{4095}{16n^4\pi^4}\cdot\frac{1}{y^4}\\
&
+\frac{13545}{32n^5\pi^5}\cdot\frac{1}{y^5}+\frac{14175}{32n^6\pi^6}\cdot\frac{1}{y^6}
+\frac{14175}{64n^7\pi^7}\cdot\frac{1}{y^7}
\Big)
\endaligned\end{equation}

\end{lemma}
By Lemma \ref{LemmaBtwo}, one has the following Lemmas \ref{LemmaAA17} and \ref{LemmaAA18} directly.
\begin{lemma}\label{LemmaAA17} For $y>0$, it holds that
\begin{equation}\aligned\nonumber
0<\frac{\partial^2}{\partial y^2}\{
\sqrt{y}K_{\frac{5}{2}}(2n\pi y)\}
<\frac{\partial^2}{\partial y^2}\{
\sqrt{y}K_{\frac{11}{2}}(2n\pi y)\}.
\endaligned\end{equation}

\end{lemma}

\begin{lemma}\label{LemmaAA18} For $1\leq m<n\in\mathbb{Z}^+$, it holds that
\begin{equation}\aligned\nonumber
\frac{\frac{\partial^2}{\partial y^2}\{
\sqrt{y}K_{\frac{5}{2}}(2n\pi y)\}}{\frac{\partial^2}{\partial y^2}\{
\sqrt{y}K_{\frac{5}{2}}(2m\pi y)\}}\leq(\frac{n}{m})^{\frac{3}{2}}e^{-2\pi(n-m)y}\\
\frac{\frac{\partial^2}{\partial y^2}\{
\sqrt{y}K_{\frac{11}{2}}(2n\pi y)\}}{\frac{\partial^2}{\partial y^2}\{
\sqrt{y}K_{\frac{11}{2}}(2m\pi y)\}}\leq(\frac{n}{m})^{\frac{3}{2}}e^{-2\pi(n-m)y}
\endaligned\end{equation}

\end{lemma}

To analyze $\frac{\partial ^2}{\partial y^2}\Big(
\zeta(6,z)-b\zeta(3,z)
\Big)$, by Lemma \ref{Lemmayy} and for convenience, we denote that
\begin{equation}\aligned\label{AAA}
\mathcal{A}_0(y;b):&=60\xi(12)y^4+60\sqrt{\pi}\frac{\Gamma(\frac{11}{2})}{\Gamma(6)}\xi(11)y^{-7}
-b\big(12\xi(6)y+12\sqrt{\pi}\frac{\Gamma(\frac{5}{2})}{\Gamma(3)}\xi(5)y^{-4}
\big)\\
\mathcal{A}_j(y;b):&=\Big(
n^{\frac{11}{2}}\sigma_{-11}(n)\frac{8\pi^6}{\Gamma(6)}\frac{\partial ^2}{\partial y^2}\{\sqrt{y}K_{\frac{11}{2}}(2\pi ny)\}
-bn^{\frac{5}{2}}\sigma_{-5}(n)\frac{8\pi^3}{\Gamma(3)}\frac{\partial ^2}{\partial y^2}\{\sqrt{y}K_{\frac{5}{2}}(2\pi ny)\}
\Big),\;j\geq1\\
\mathcal{A}_0(y):&=\mathcal{A}_0(y;3.06);\;\;\\
\mathcal{A}_j(y):&=\mathcal{A}_j(y;3.06), j\geq1,
\endaligned\end{equation}
and
\begin{equation}\aligned\nonumber
\mathcal{A}_{j,+}(y):&=n^{\frac{11}{2}}\sigma_{-11}(n)\frac{8\pi^6}{\Gamma(6)}\frac{\partial ^2}{\partial y^2}\{\sqrt{y}K_{\frac{11}{2}}(2\pi ny)\};\\
\mathcal{A}_{j,-}(y):&=3.06n^{\frac{5}{2}}\sigma_{-5}(n)\frac{8\pi^3}{\Gamma(3)}\frac{\partial ^2}{\partial y^2}\{\sqrt{y}K_{\frac{5}{2}}(2\pi ny)\}, j\geq1.
\endaligned\end{equation}
Then
\begin{equation}\aligned\nonumber
\mathcal{A}_j(y)=\mathcal{A}_{j,+}(y)-\mathcal{A}_{j,-}(y),\;\; j\geq1.
\endaligned\end{equation}
One first has a preliminary estimate by Lemma \ref{LemmaAA17}
\begin{lemma}\label{LemmaA2n}
 $\mathcal{A}_n(x)>0$
if $n\geq2$ and $y>0$.
\end{lemma}
By notations in \eqref{AAA} and Lemma \ref{Lemmayy}, one has
 \begin{equation}\aligned\nonumber
\frac{\partial ^2}{\partial y^2}\Big(
\zeta(6,z)-b\zeta(3,z)
\Big)
=\mathcal{A}_0(y)
+\sum_{n=1}^\infty\mathcal{A}_n(y)\cdot\cos(2\pi n x).
\endaligned\end{equation}
By notations in \eqref{AAA}, the prove of Lemma \ref{LemmaA8} equivalents to the following
\begin{lemma}[{\bf y-convexity on a particular value: alternative version}]\label{LemmaA9}
\begin{equation}\aligned\nonumber
\mathcal{A}_0(y)
+\sum_{n=1}^\infty\mathcal{A}_j(y)\cdot\cos(2\pi n x)
>0,\;\;\hbox{for}\;\; z\in\overline{\mathcal{D}_{\mathcal{G}}}.
\endaligned\end{equation}
\end{lemma}
Note that $z\in\overline{\mathcal{D}_{\mathcal{G}}}$ implies that $x\in[0,\frac{1}{2}]$
and $y\geq\frac{\sqrt3}{2}$. In fact we will show that
\begin{lemma}[{\bf y-convexity on a infinite rectangular region}]\label{LemmaA9a}
\begin{equation}\aligned\nonumber
\mathcal{A}_0(y)
+\sum_{n=1}^\infty\mathcal{A}_j(y)\cdot\cos(2\pi n x)
>0,\;\;\hbox{for}\;\; \{x\in[0,\frac{1}{2}]\}\cap \{y\in[\frac{\sqrt3}{2},\infty)\}.
\endaligned\end{equation}
\end{lemma}

We first estimate directly that

\begin{lemma}\label{LemmaA10}
\begin{equation}\aligned\nonumber
\mathcal{A}_1(y)-\sum_{n=2}^\infty n^2 \mathcal{A}_n(y)\geq3\;\;\hbox{for}\;\; y\in[\frac{\sqrt3}{2},1].
\endaligned\end{equation}
\end{lemma}

Then by a direct computation we have that
\begin{lemma}\label{LemmaA11}
\begin{equation}\aligned\nonumber
\mathcal{A}_0(y)+\sum_{n=1}^\infty (-1)^n\mathcal{A}_n(y)\geq4.688\cdot10^{-3}\;\;\hbox{for}\;\; y\in[\frac{\sqrt3}{2},1].
\endaligned\end{equation}
\end{lemma}

Lemmas \ref{LemmaA10} and \ref{LemmaA11} yields that
\begin{lemma}$[$Estimate on $y\in[\frac{\sqrt3}{2},1]$  $]$\label{LemmaA12}
\begin{equation}\aligned\nonumber
\mathcal{A}_0(y)
+\sum_{n=1}^\infty\mathcal{A}_j(y)\cdot\cos(2\pi n x)
>0,\;\;\hbox{for}\;\; \{x\in[0,\frac{1}{2}]\}\cap \{y\in[\frac{\sqrt3}{2},1]\}.
\endaligned\end{equation}
\end{lemma}

\begin{proof}[Proof of Lemma \ref{LemmaA12}]
Since $x\in[0,\frac{1}{2}]$, by Lemmas \ref{LemmaA10}, \ref{LemmaA2n} and \ref{LemmaSin},
\begin{equation}\aligned\label{P1}
\frac{\partial}{\partial x}\Big(\mathcal{A}_0(y)
+\sum_{n=1}^\infty\mathcal{A}_j(y)\cdot\cos(2\pi n x)
\Big)
&=-2\pi\sin(2\pi x)\cdot \sum_{n=1}^\infty n\mathcal{A}_n(y)\frac{\sin(2\pi n x)}{\sin(2\pi x)}\\
&\leq-2\pi\sin(2\pi x)\cdot\Big(\mathcal{A}_1(y)-n^2\mathcal{A}_n(y)
\Big)\\
&\leq0.
\endaligned\end{equation}

It follows by \eqref{P1} and Lemma \ref{LemmaA10} that, for $\{x\in[0,\frac{1}{2}]\}\cap \{y\in[\frac{\sqrt3}{2},1]\}$,
\begin{equation}\aligned\label{P1}
\Big(\mathcal{A}_0(y)
+\sum_{n=1}^\infty\mathcal{A}_j(y)\cdot\cos(2\pi n x)
\Big)
&\geq\Big(\mathcal{A}_0(y)
+\sum_{n=1}^\infty\mathcal{A}_j(y)\cdot\cos(2\pi n x)
\Big)\mid_{x=\frac{1}{2}}\\
&=\mathcal{A}_0(y)+\sum_{n=1}^\infty (-1)^n\mathcal{A}_n(y)\\
&>0.
\endaligned\end{equation}

\end{proof}

\begin{lemma}\label{LemmaA13} For $y\geq1$,
\begin{equation}\aligned\nonumber
\mathcal{A}_0(y)-\sum_{n=1}^\infty \mid\mathcal{A}_n(y)\mid\geq5>0.
\endaligned\end{equation}

\end{lemma}
\begin{proof}[ Proof of Lemma \ref{LemmaA13}]
One first has the preliminary estimates
\begin{equation}\aligned\label{LA13a}
\mathcal{A}_0(y)&\geq\mathcal{A}_0(1)=24.211215\cdots,\\
|\mathcal{A}_1(y)|&\leq\mathcal{A}_1(1)=10.070150\cdots\;\;\hbox{for}\;\; y\geq1.
\endaligned\end{equation}
Then for the remainder terms, by Lemma \ref{LemmaA2n}, one has
\begin{equation}\aligned\label{LA13b}
\sum_{n=2}^\infty \mid\mathcal{A}_n(y)
&\leq\sum_{n=2}^\infty n^{\frac{11}{2}}\sigma_{-11}(n)\frac{8\pi^6}{\Gamma(6)}
\frac{\partial^2}{\partial y^2}\{\sqrt yK_{\frac{11}{2}}(2\pi ny)\}\\
&\leq\frac{8\pi^6}{\Gamma(6)}
\frac{\partial^2}{\partial y^2}\{\sqrt yK_{\frac{11}{2}}(2\pi y)\}
\sum_{n=2}^\infty n^{\frac{11}{2}}\sigma_{-11}(n)
\cdot
\frac{\frac{\partial^2}{\partial y^2}\{\sqrt yK_{\frac{11}{2}}(2\pi ny)\}}{\frac{\partial^2}{\partial y^2}\{\sqrt yK_{\frac{11}{2}}(2\pi ny)\}}\\
&\leq\frac{8\pi^6}{\Gamma(6)}
\frac{\partial^2}{\partial y^2}\{\sqrt yK_{\frac{11}{2}}(2\pi y)\}\mid_{y=1}\cdot
\sum_{n=2}^\infty n^{7}\sigma_{-11}(n)
\cdot e^{-2\pi(n-1)y}\\
&\leq8.666969022,\;\;\hbox{for}\;\; y\geq1.
\endaligned\end{equation}
Here one used that Lemmas \ref{LemmaAA18} and \ref{LemmaKKK} and
\begin{equation}\aligned\nonumber
\frac{8\pi^6}{\Gamma(6)}
\frac{\partial^2}{\partial y^2}\{\sqrt yK_{\frac{11}{2}}(2\pi y)\}\mid_{y=1}=35.105436\cdots\;\;\hbox{and}\;\;
\sum_{n=2}^\infty n^{7}\sigma_{-11}(n)
\cdot e^{-2\pi(n-1)y}=0.246883\cdots.
\endaligned\end{equation}
The result follows by \eqref{LA13a} and \eqref{LA13b}.

\end{proof}

Lemma \ref{LemmaA13} implies that the following estimate for  $y \geq 1 $
\begin{lemma}$[$Estimate on $y\in[1,\infty)$  $]$\label{LemmaA14}
\begin{equation}\aligned\nonumber
\mathcal{A}_0(y)
+\sum_{n=1}^\infty\mathcal{A}_j(y)\cdot\cos(2\pi n x)
\geq5>0,\;\;\hbox{for}\;\; \{x\in[0,\frac{1}{2}]\}\cap \{y\in[1,\infty)\}.
\endaligned\end{equation}
\end{lemma}

\begin{proof}[ Proof of Lemma \ref{LemmaA14}] It follows by the following observation
\begin{equation}\aligned\nonumber
\mathcal{A}_0(y)
+\sum_{n=1}^\infty\mathcal{A}_j(y)\cdot\cos(2\pi n x)
&\geq
\mathcal{A}_0(y)-\sum_{n=1}^\infty \mid\mathcal{A}_n(y)\mid\\
&\geq5>0.
\endaligned\end{equation}

\end{proof}

Lemmas \ref{LemmaA12} and \ref{LemmaA14} complete the proof of Lemma \ref{LemmaA9a} and hence Lemma \ref{LemmaA9}.

\subsection{Proof of Proposition \ref{Lemma102}}

In this sub-section, we prove the mixed order derivative estimate in Proposition \ref{Lemma102}. The strategy of the proof is similar to that of Proposition \ref{Lemma101}. First we have the following comparison principle.

\begin{lemma}\label{LemmaB1} If there exists $b=\overline{b}>0$ such that
\begin{equation}\aligned\label{C1}
\frac{\partial^2}{\partial y\partial x}\Big(\zeta(6,z)-b\zeta(3,z)\Big)
>0,\;\;\hbox{for}\;\; z\in\{x\in[0,\frac{1}{2}]\}\cap \{y\in[\frac{\sqrt3}{2},1]\}.
\endaligned\end{equation}
Then for all $b\leq\overline{b}$, \eqref{C1} holds.
\end{lemma}
By Lemma \ref{LemmaB1}, one only needs to consider a particular point of the parameter $b$.
The proof of Lemma \ref{LemmaB1} is based on
\begin{equation}\aligned\nonumber
\frac{\partial}{\partial b}\Big(\frac{\partial^2}{\partial y\partial x}\Big(\zeta(6,z)-b\zeta(3,z)\Big)\Big)
=-\frac{\partial^2}{\partial y\partial x}\zeta(3,z)
\endaligned\end{equation}
and the following lemma
\begin{lemma}[{\bf mixed derivative of $\zeta(3,z)$ on a rectangular region}]\label{LemmaB2}
\begin{equation}\aligned\label{C2}
\frac{\partial^2}{\partial y\partial x}\zeta(3,z)\geq0,\;\;\hbox{for}\;\; z=x+iy\in\{x\in[0,\frac{1}{2}]\}\cap \{y\in[\frac{\sqrt3}{2},1])\}.
\endaligned\end{equation}

\end{lemma}

In fact, we shall prove that
\begin{lemma}[{\bf mixed derivative of $\zeta(3,z)$ on a infinite rectangular region}]\label{LemmaB3}
\begin{equation}\aligned\label{C2}
\frac{\partial^2}{\partial y\partial x}\zeta(3,z)\geq0,\;\;\hbox{for}\;\; z=x+iy\in\{x\in[0,\frac{1}{2}]\}\cap \{y\in[\frac{\sqrt3}{2},\infty)\}.
\endaligned\end{equation}

\end{lemma}

\begin{proof}[ Proof of Lemma \ref{LemmaB3}]
We start by Theorem \ref{ThCS}
\begin{equation}\aligned\label{C21}
\frac{\partial^2}{\partial y\partial x}\zeta(3,z)=
2\pi\sum_{n=1}^\infty n^{\frac{7}{2}}\sigma_{-5}(n)\frac{8\pi^3}{\Gamma(3)}\big(-\frac{\partial}{\partial y}\{\sqrt yK_{\frac{5}{2}}(2\pi ny)\}\big)\cdot\sin(2\pi n x).
\endaligned\end{equation}
We continue by listing the factors and rewrite it by quotient forms
\begin{equation}\aligned\label{C22}
\frac{\partial^2}{\partial y\partial x}\zeta(3,z)=
\frac{16\pi^4}{\Gamma(3)}\sin(2\pi x)\big(-\frac{\partial}{\partial y}\{\sqrt yK_{\frac{5}{2}}(2\pi y)\}\big)\sum_{n=1}^\infty n^{\frac{7}{2}}\sigma_{-5}(n)
\frac{\frac{\partial}{\partial y}\{\sqrt yK_{\frac{5}{2}}(2\pi ny)\}}{\frac{\partial}{\partial y}\{\sqrt yK_{\frac{5}{2}}(2\pi y)\}}\cdot\frac{\sin(2\pi n x)}{\sin(2\pi  x)}.
\endaligned\end{equation}
 By Lemmas \ref{LemmaB6} and \ref{LemmaB7}, separating the term of $n=1$ and the terms of $n\geq2$, one has
\begin{equation}\aligned\label{C23}
\frac{\partial^2}{\partial y\partial x}\zeta(3,z)&\geq
\frac{16\pi^4}{\Gamma(3)}\sin(2\pi x)\cdot\big(-\frac{\partial}{\partial y}\{\sqrt yK_{\frac{5}{2}}(2\pi y)\}\big)
\cdot\Big(
1-\sum_{n=2}^\infty n^{5}\sigma_{-5}(n)\cdot e^{-2\pi(n-1)y}
\Big)\\
&\geq\frac{8\pi^4}{\Gamma(3)}\sin(2\pi x)\cdot\big(-\frac{\partial}{\partial y}\{\sqrt yK_{\frac{5}{2}}(2\pi y)\}\big)\\
&\geq0.
\endaligned\end{equation}
Here we used that since $y\geq\frac{\sqrt3}{2}$
\begin{equation}\aligned\label{C24}
\sum_{n=2}^\infty n^{5}\sigma_{-5}(n)\cdot e^{-2\pi(n-1)y}&\leq\sum_{n=2}^\infty n^{5}\sigma_{-5}(n)\cdot e^{-\sqrt{3}\pi(n-1)}\\
&\cong0.147795\cdots<\frac{1}{2}
\endaligned\end{equation}
and Lemma \ref{LemmaSin}.
The proof is complete.
\end{proof}

We proceed by Lemma \ref{LemmaCS}, one then has
\begin{lemma}\label{LemmaB5}
\begin{equation}\aligned\nonumber
\frac{\partial}{\partial x}\frac{\partial}{\partial y}
\Big(
\zeta(6,z)-\zeta(3,z)
\Big)
&=2\pi\sin(2\pi x)\cdot
\Big(
\sum_{n=1}^\infty\big( n^{\frac{13}{2}}\sigma_{-11}(n)\frac{8\pi^6}{\Gamma(6)}(-\frac{\partial}{\partial y}\{\sqrt yK_{\frac{11}{2}}(2\pi ny)\})\\
&\;\;\;-bn^{\frac{7}{2}}\sigma_{-5}(n)\frac{8\pi^3}{\Gamma(3)}(-\frac{\partial}{\partial y}\{\sqrt yK_{\frac{5}{2}}(2\pi ny)\})\big)
\cdot\frac{\sin(2\pi nx)}{\sin(2\pi x)}
\Big).
\endaligned\end{equation}

\end{lemma}
By Lemma \ref{LemmaB1}, to prove Proposition \ref{Lemma102}, it suffices to prove the case when $b=3.07$.
Before going to the proof, we introduce some facts on $-\frac{\partial}{\partial y}\{
\sqrt{y}K_{\frac{5}{2}}(2n\pi y)\}$ and $-\frac{\partial}{\partial y}\{
\sqrt{y}K_{\frac{11}{2}}(2n\pi y)\}$.
By Proposition \ref{PropW}, one has the following

\begin{lemma}\label{Lemmay}Two basic identities hold
\begin{equation}\aligned\nonumber
-\frac{\partial}{\partial y}\{
\sqrt{y}K_{\frac{5}{2}}(2n\pi y)\}
=&\sqrt{n}\pi e^{-2\pi ny}
\Big(
1+\frac{3}{2n\pi}\cdot\frac{1}{y}+\frac{3}{2n^2\pi^2}\cdot\frac{1}{y^2}
+\frac{3}{4n^3\pi^3}\cdot\frac{1}{y^3}
\Big)\\
-\frac{\partial}{\partial y}\{
\sqrt{y}K_{\frac{11}{2}}(2n\pi y)\}
=&\sqrt{n}\pi e^{-2\pi ny}
\Big(
1+\frac{15}{n\pi}\cdot\frac{1}{y}+\frac{30}{n^2\pi^2}\cdot\frac{1}{y^2}
+\frac{315}{4n^3\pi^3}\cdot\frac{1}{y^3}+\frac{2205}{16n^4\pi^4}\cdot\frac{1}{y^4}\\
&
+\frac{4725}{32n^5\pi^5}\cdot\frac{1}{y^5}+\frac{4725}{64n^6\pi^6}\cdot\frac{1}{y^6}
\Big)
\endaligned\end{equation}

\end{lemma}
Lemmas \ref{LemmaB6} and \ref{LemmaB7} follow directly from Lemma \ref{Lemmay}.
\begin{lemma}\label{LemmaB6} For $y>0$, it holds that
\begin{equation}\aligned\nonumber
0<-\frac{\partial}{\partial y}\{
\sqrt{y}K_{\frac{5}{2}}(2n\pi y)\}
<-\frac{\partial}{\partial y}\{
\sqrt{y}K_{\frac{11}{2}}(2n\pi y)\}.
\endaligned\end{equation}

\end{lemma}

\begin{lemma}\label{LemmaB7} For $1\leq m<n\in\mathbb{Z}^+$, it holds that
\begin{equation}\aligned\nonumber
&\frac{\frac{\partial}{\partial y}\{
\sqrt{y}K_{\frac{5}{2}}(2n\pi y)\}}{\frac{\partial}{\partial y}\{
\sqrt{y}K_{\frac{5}{2}}(2m\pi y)\}}\leq\sqrt{\frac{n}{m}}e^{-2\pi(n-m)y},\\
&\frac{\frac{\partial}{\partial y}\{
\sqrt{y}K_{\frac{11}{2}}(2n\pi y)\}}{\frac{\partial}{\partial y}\{
\sqrt{y}K_{\frac{11}{2}}(2m\pi y)\}}\leq\sqrt{\frac{n}{m}}e^{-2\pi(n-m)y}.
\endaligned\end{equation}

\end{lemma}

For convenience, we denote that
\begin{equation}\aligned\label{B100}
\mathcal{B}_n(y):=\big( n^{\frac{13}{2}}\sigma_{-11}(n)\frac{8\pi^6}{\Gamma(6)}(-\frac{\partial}{\partial y}\{\sqrt yK_{\frac{11}{2}}(2\pi ny)\})-3.07n^{\frac{7}{2}}\sigma_{-5}(n)\frac{8\pi^3}{\Gamma(3)}(-\frac{\partial}{\partial y}\{\sqrt yK_{\frac{5}{2}}(2\pi ny)\})\big)
\endaligned\end{equation}
A direct observation on $\mathcal{B}_n(y)$ followed by Lemma \ref{LemmaB6} is
\begin{lemma}\label{LemmaBn}
\begin{equation}\aligned\nonumber
\mathcal{B}_n(y)>0\;\;\hbox{for}\;\;n\geq2,\; y>0.
\endaligned\end{equation}
\end{lemma}

Then by Lemma \ref{LemmaB5}
\begin{equation}\aligned\nonumber
\frac{\partial}{\partial x}\frac{\partial}{\partial y}
\Big(
\zeta(6,z)-3.07\zeta(3,z)
\Big)
&=2\pi\sin(2\pi x)\cdot
\sum_{n=1}^\infty\mathcal{B}_n(y)
\cdot\frac{\sin(2\pi nx)}{\sin(2\pi x)}
.
\endaligned\end{equation}
The proof of Proposition \ref{Lemma102} is reduced to the following
\begin{lemma}\label{LemmaBB}
\begin{equation}\aligned\nonumber
\sum_{n=1}^\infty\mathcal{B}_n(y)
\cdot\frac{\sin(2\pi nx)}{\sin(2\pi x)}
>0\;\;\hbox{for}\;\;\{x\in[0,\frac{1}{2}]\}\cap \{y\in[\frac{\sqrt3}{2},\infty)\}.
\endaligned\end{equation}
Here $\mathcal{B}_n(y)$ is defined in \eqref{B100}.
\end{lemma}

We shall analyze the main order terms $\sum_{n=1}^3\mathcal{B}_n(y)
\cdot\frac{\sin(2\pi nx)}{\sin(2\pi x)}$ and the error terms
$\sum_{n=4}^\infty\mathcal{B}_n(y)
\cdot\frac{\sin(2\pi nx)}{\sin(2\pi x)}$ respectively in the following.

We first simplify the main order terms $\sum_{n=1}^3\mathcal{B}_n(y)
\cdot\frac{\sin(2\pi nx)}{\sin(2\pi x)}$
\begin{lemma}\label{LemmaM67}
\begin{equation}\aligned\nonumber
\sum_{n=1}^3\mathcal{B}_n(y)\cdot\frac{\sin(2\pi nx)}{\sin(2\pi x)}
=\mathcal{B}_1(y)+\mathcal{B}_3(y)+
2\mathcal{B}_2(y)\cos(2\pi x)
+2\mathcal{B}_3(y)\cos(4\pi x).
\endaligned\end{equation}
Here $\mathcal{B}_n(y)$ is defined in \eqref{B100}.
\end{lemma}
By Lemma \ref{LemmaM67},
\begin{equation}\aligned\label{L100}
\frac{\partial}{\partial x}\Big(\sum_{n=1}^3\mathcal{B}_n(y)\cdot\frac{\sin(2\pi nx)}{\sin(2\pi x)}\Big)
=-2\sin(2\pi x)\cdot\big(\mathcal{B}_2(y)+2\mathcal{B}_3(y)\cos(2\pi x)\big).
\endaligned\end{equation}
We shall prove that

\begin{lemma}\label{LemmaM68}For $\{x\in[0,\frac{1}{2}]\}\cap \{y\in[\frac{\sqrt3}{2},\infty)\}$, it holds that
\begin{equation}\aligned\nonumber
\frac{\partial}{\partial x}\Big(\sum_{n=1}^3\mathcal{B}_n(y)\cdot\frac{\sin(2\pi nx)}{\sin(2\pi x)}\Big)
\leq0.
\endaligned\end{equation}

\end{lemma}
By \eqref{L100}, the proof of Lemma \ref{LemmaM68} is based on following lemma

\begin{lemma}\label{Lemma4.28b}
\begin{equation}\aligned\nonumber
\mathcal{B}_2(y)-2\mathcal{B}_3(y)\geq\frac{1}{10}\cdot
2^{\frac{13}{2}}\sigma_{-11}(2)\frac{8\pi^6}{\Gamma(6)}(-\frac{\partial}{\partial y}\{\sqrt yK_{\frac{11}{2}}(4\pi y)\})>0
\;\;\hbox{for}\;\;y\geq\frac{\sqrt3}{2}.
\endaligned\end{equation}
Here $\mathcal{B}_n(y)$ is defined in \eqref{B100}.
\end{lemma}

\begin{proof}[ Proof of Lemma \ref{Lemma4.28b}] By \eqref{B100}, one has the explicit expression,
\begin{equation}\aligned\nonumber
\mathcal{B}_2(y)-2\mathcal{B}_3(y)
= 2^{\frac{13}{2}}\sigma_{-11}(2)\frac{8\pi^6}{\Gamma(6)}(-\frac{\partial}{\partial y}\{\sqrt yK_{\frac{11}{2}}(4\pi y)\})-3.07\cdot 2^{\frac{7}{2}}\sigma_{-5}(2)\frac{8\pi^3}{\Gamma(3)}(-\frac{\partial}{\partial y}\{\sqrt yK_{\frac{5}{2}}(4\pi y)\})\\
-2\cdot 3^{\frac{13}{2}}\sigma_{-11}(3)\frac{8\pi^6}{\Gamma(6)}(-\frac{\partial}{\partial y}\{\sqrt yK_{\frac{11}{2}}(6\pi y)\})+2\cdot3.07\cdot 3^{\frac{7}{2}}\sigma_{-5}(3)\frac{8\pi^3}{\Gamma(3)}(-\frac{\partial}{\partial y}\{\sqrt yK_{\frac{5}{2}}(6\pi y)\})\big).
\endaligned\end{equation}
Denote that
\begin{equation}\aligned\nonumber
\mathcal{B}_c:=
2^{\frac{13}{2}}\sigma_{-11}(2)\frac{8\pi^6}{\Gamma(6)}(-\frac{\partial}{\partial y}\{\sqrt yK_{\frac{11}{2}}(4\pi y)\})
\endaligned\end{equation}
for convenience.
We then drop the last positive term and take leading order by
\begin{equation}\aligned\label{Eq999}
\mathcal{B}_2(y)-2\mathcal{B}_3(y)
&\geq\mathcal{B}_c\cdot
\Big(
1-3.07\frac{\sigma_{-5}(2)}{\sigma_{-11}(2)}\frac{\Gamma(6)}{8\pi^3\Gamma(3)}\frac{\frac{\partial}{\partial y}\{\sqrt yK_{\frac{5}{2}}(4\pi y)\}}{\frac{\partial}{\partial y}\{\sqrt yK_{\frac{11}{2}}(4\pi y)\}}
-2(\frac{3}{2})^{\frac{13}{2}}\frac{\sigma_{-11}(2)}{\sigma_{-11}(3)}\frac{\frac{\partial}{\partial y}\{\sqrt yK_{\frac{11}{2}}(6\pi y)\}}{\frac{\partial}{\partial y}\{\sqrt yK_{\frac{11}{2}}(4\pi y)\}}
\Big)\\
&\geq\mathcal{B}_c\cdot
\Big(
1-3.07\frac{\sigma_{-5}(2)}{\sigma_{-11}(2)}\frac{\Gamma(6)}{8\pi^3\Gamma(3)}
-2(\frac{3}{2})^{\frac{13}{2}}\frac{\sigma_{-11}(2)}{\sigma_{-11}(3)}\cdot e^{-2\pi y}\Big).
\endaligned\end{equation}
Here one used that
\begin{equation}\aligned\nonumber
\frac{\frac{\partial}{\partial y}\{\sqrt yK_{\frac{5}{2}}(4\pi y)\}}{\frac{\partial}{\partial y}\{\sqrt yK_{\frac{11}{2}}(4\pi y)\}}\leq1
\;\;\hbox{and}\;\;
\frac{\frac{\partial}{\partial y}\{\sqrt yK_{\frac{11}{2}}(6\pi y)\}}{\frac{\partial}{\partial y}\{\sqrt yK_{\frac{11}{2}}(4\pi y)\}}
\leq e^{-2\pi y}
\endaligned\end{equation}
by Lemmas \ref{LemmaB6} and \ref{LemmaB7}.
The proof is complete by \eqref{Eq999} and the following
\begin{equation}\aligned\nonumber
3.07\frac{\sigma_{-5}(2)}{\sigma_{-11}(2)}\frac{\Gamma(6)}{8\pi^3\Gamma(3)}\leq0.7654238277\cdots\;\;\hbox{and}\;\;
2(\frac{3}{2})^{\frac{13}{2}}\frac{\sigma_{-11}(2)}{\sigma_{-11}(3)}\cdot e^{-\sqrt3\pi }\leq0.1208493896\cdots.
\endaligned\end{equation}

\end{proof}

By Lemmas \ref{LemmaM68} and \ref{LemmaM67}, one has
\begin{lemma}\label{LemmaM610}For $\{x\in[0,\frac{1}{2}]\}\cap \{y\in[\frac{\sqrt3}{2},\infty)\}$, it holds that
\begin{equation}\aligned\nonumber
\sum_{n=1}^3\mathcal{B}_n(y)\cdot\frac{\sin(2\pi nx)}{\sin(2\pi x)}
\geq\mathcal{B}_1(y)-2\mathcal{B}_2(y)+3\mathcal{B}_3(y).
\endaligned\end{equation}
Here $\mathcal{B}_n(y)$ is defined in \eqref{B100}.
\end{lemma}
For controlling the error terms, we denote that
\begin{equation}\aligned\label{B101}
\epsilon_2(y):=\sum_{n=4}^\infty n^{\frac{15}{2}}\sigma_{-11}(n)\frac{8\pi^6}{\Gamma(6)}(-\frac{\partial}{\partial y}\{\sqrt yK_{\frac{11}{2}}(2\pi ny)\}).
\endaligned\end{equation}

We then estimate the error terms by $\epsilon_2(y)$
\begin{lemma}\label{LemmaB610} For $y\geq\frac{\sqrt3}{2}$,
\begin{equation}\aligned\nonumber
\sum_{n=4}^\infty\mathcal{B}_n(y)\cdot\frac{\sin(2\pi nx)}{\sin(2\pi x)}
\geq
\begin{cases}
0,\;\;\;\;\;\;\;\;\;\;x\in[0,\frac{1}{8}]\\
-\epsilon_2(y),\;\;x\in(\frac{1}{8},\frac{1}{2}].
\end{cases}
\endaligned\end{equation}
Here $\mathcal{B}_n(y)$ and $\epsilon(y)$ are defined in \eqref{B100} and \eqref{B101} respectively.
\end{lemma}
\begin{proof}[ Proof of Lemma \ref{LemmaB610}] Note that by Lemma \ref{LemmaBn}, for $n\geq2$, $\mathcal{B}_n(y)>0$ for $y>0$.

The case $x\in(\frac{1}{8},\frac{1}{2}]$: by Lemma \ref{LemmaSin},
\begin{equation}\aligned\label{KKa}
\sum_{n=4}^\infty|\mathcal{B}_n(y)\cdot\frac{\sin(2\pi nx)}{\sin(2\pi x)}|
\leq\sum_{n=4}^\infty n\cdot|\mathcal{B}_n(y)|.
\endaligned\end{equation}
Then by Lemma \ref{LemmaB6}, for $n\geq4$,
\begin{equation}\aligned\label{KKb}
|\mathcal{B}_n(y)|\leq n^{\frac{13}{2}}\sigma_{-11}(n)\frac{8\pi^6}{\Gamma(6)}(-\frac{\partial}{\partial y}\{\sqrt yK_{\frac{11}{2}}(2\pi ny)\}).
\endaligned\end{equation}
Therefore the result follows by \eqref{KKa} and \eqref{KKb}.

The case $x\in[0,\frac{1}{8}]$: in this case, $\frac{\sin(8\pi x)}{\sin(2\pi x)}\geq0$. By Lemmas \ref{LemmaSin}, \ref{LemmaBn} and \ref{LemmaB7},
\begin{equation}\aligned\nonumber
\sum_{n=4}^\infty\mathcal{B}_n(y)\cdot\frac{\sin(2\pi nx)}{\sin(2\pi x)}
&=\mathcal{B}_4(y)\frac{\sin(8\pi x)}{\sin(2\pi x)}\cdot
\Big(
1-\sum_{n=5}^\infty \frac{\sin(2\pi nx)}{\sin(8\pi x)}\cdot\frac{\mathcal{B}_n(y)}{\mathcal{B}_4(y)}
\Big)\\
&\geq\mathcal{B}_4(y)\frac{\sin(8\pi x)}{\sin(2\pi x)}\cdot
\Big(
1-\sum_{n=5}^\infty \frac{n}{4}\cdot\frac{\mathcal{B}_n(y)}{\mathcal{B}_4(y)}
\Big)\\
&\geq\mathcal{B}_4(y)\frac{\sin(8\pi x)}{\sin(2\pi x)}\cdot
\Big(
1-2\sum_{n=5}^\infty \frac{n}{4}\cdot\frac{n^{\frac{13}{2}}\sigma_{-11}(n)\frac{8\pi^6}{\Gamma(6)}(-\frac{\partial}{\partial y}\{\sqrt yK_{\frac{11}{2}}(2\pi ny)\})}{4^{\frac{13}{2}}\sigma_{-11}(4)\frac{8\pi^3}{\Gamma(3)}(-\frac{\partial}{\partial y}\{\sqrt yK_{\frac{11}{2}}(8\pi y)\})}
\Big)\\
&\geq\mathcal{B}_4(y)\frac{\sin(8\pi x)}{\sin(2\pi x)}\cdot
\Big(
1-2\frac{\frac{8\pi^6}{\Gamma(6)}}{\frac{8\pi^3}{\Gamma(3)}}\sum_{n=5}^\infty\sigma_{-11}(n)(\frac{n}{4})^8\cdot e^{-2\pi(n-1)y}
\Big)\\
&\geq\frac{1}{2}\mathcal{B}_4(y)\frac{\sin(8\pi x)}{\sin(2\pi x)}\\
&\geq0.
\endaligned\end{equation}
Here one used that
\begin{equation}\aligned\nonumber
\frac{\frac{8\pi^6}{\Gamma(6)}}{\frac{8\pi^3}{\Gamma(3)}}\sum_{n=5}^\infty\sigma_{-11}(n)(\frac{n}{4})^8\cdot e^{-2\pi(n-1)y}
\leq0.028<\frac{1}{4},\;\;\hbox{for}\;\;y\geq\frac{\sqrt3}{2}.
\endaligned\end{equation}

\end{proof}
By Lemmas \ref{LemmaM610} and \ref{LemmaB610}, one has
\begin{lemma}\label{LemmaMMBB}For $\{x\in[0,\frac{1}{2}]\}\cap \{y\in[\frac{\sqrt3}{2},\infty)\}$, it holds that
\begin{equation}\aligned\nonumber
\sum_{n=1}^\infty\mathcal{B}_n(y)\cdot\frac{\sin(2\pi nx)}{\sin(2\pi x)}
\geq\mathcal{B}_1(y)-2\mathcal{B}_2(y)+3\mathcal{B}_3(y)-\epsilon_2(y).
\endaligned\end{equation}
Here $\mathcal{B}_n(y)$ and $\epsilon_2(y)$ are defined in \eqref{B100} and \eqref{B101} respectively.
\end{lemma}

By Lemma \ref{LemmaMMBB}, the proof of Lemma \ref{LemmaBB} reduces to
\begin{lemma} For $y\in[\frac{\sqrt3}{2},1]$, it holds that
\begin{equation}\aligned\nonumber
\Big(\mathcal{B}_1(y)-2\mathcal{B}_2(y)+3\mathcal{B}_3(y)-\epsilon_2(y)\Big)>0.
\endaligned\end{equation}
Here $\mathcal{B}_n(y)$ and $\epsilon_2(y)$ are defined in \eqref{B100} and \eqref{B101} respectively.
\end{lemma}
Indeed, one shows by direct computation that
\begin{lemma}\label{Lemma434}
For $y\in[\frac{\sqrt3}{2},1]$, it holds that
\begin{equation}\aligned\nonumber
&\Big(\mathcal{B}_1(y)-2\mathcal{B}_2(y)+3\mathcal{B}_3(y)-\epsilon_2(y)\Big)\\
\geq&\Big(\mathcal{B}_1(y)-2\mathcal{B}_2(y)+3\mathcal{B}_3(y)
-\epsilon_2(y)\Big)\mid_{y=1}
=5.87\cdot 10^{-3}>0.
\endaligned\end{equation}

Here $\mathcal{B}_n(y)$ and $\epsilon_2(y)$ are defined in \eqref{B100} and \eqref{B101} respectively.
\end{lemma}

\section{Minimizers in the remaining cases: $b\in(3.06,3.062)$}
Recall that
\begin{equation}\aligned\nonumber
\Gamma_a&=\{
z\in\mathbb{H}: \Re(z)=0,\; \Im(z)\geq1
\};\\
\Gamma_b&=\{
z\in\mathbb{H}: |z|=1,\; \Im(z)\in[\frac{\sqrt3}{2},1]
\}.
\endaligned\end{equation}
And
\begin{equation}\aligned\nonumber
\Gamma=\Gamma_a\cup\Gamma_b.
\endaligned\end{equation}

In Theorems \ref{ThA} and \ref{Th102}, we have located the minimizers on the curve $\Gamma$ and the $\frac{1}{4}-$arc, i.e., $\Gamma_b$ for $b\geq3.062$ and $b\leq3.06$ respectively. In summary, we already obtain that

\begin{theorem}[Theorems \ref{ThA} and \ref{Th102}]\label{Th111} For $b\in(-\infty,3.06]\cup[3.062,+\infty)$,
\begin{equation}\aligned\nonumber
&\min_{z\in\mathbb{H}}\Big(\zeta(6,z)-b\zeta(3,z)\Big)
=\min_{z\in\overline{\mathcal{D}_{\mathcal{G}}}}\Big(\zeta(6,z)-b\zeta(3,z)\Big)\\
=&\min_{z\in \Gamma}\Big(\zeta(6,z)-b\zeta(3,z)\Big).
\endaligned\end{equation}

\end{theorem}
Theorem \ref{Th111} is a direct consequence of Theorems \ref{ThA} and \ref{Th102}.

In this section, we aim to consider  the remaining case of $b\in(3.06,3.062)$. The following is the main result of this section:
\begin{theorem}\label{Th112}
For $b\in(3.06,3.062)$,
\begin{equation}\aligned\nonumber
\min_{z\in\mathbb{H}}\Big(\zeta(6,z)-b\zeta(3,z)\Big)
=\min_{z\in\overline{\mathcal{D}_{\mathcal{G}}}}\Big(\zeta(6,z)-b\zeta(3,z)\Big)
\;\;\hbox{is achieved at }i.
\endaligned\end{equation}
\end{theorem}

By Theorems \ref{Th111} and \ref{Th112}, we have a global picture of the minimizers of the functional $\Big(\zeta(6,z)-b\zeta(3,z)\Big)$,
namely,
\begin{theorem}\label{Th113}
For $b\in\R$,
\begin{equation}\aligned\nonumber
\min_{z\in\mathbb{H}}\Big(\zeta(6,z)-b\zeta(3,z)\Big)
=\min_{z\in\Gamma}\Big(\zeta(6,z)-b\zeta(3,z)\Big)
.
\endaligned\end{equation}
\end{theorem}
Theorem \ref{Th113} provides a unified way to locate the
minimizers of the functional $\Big(\zeta(6,z)-b\zeta(3,z)\Big)$ for all $b\in\R$. It remains to determine where $\min_{z\in\Gamma}\Big(\zeta(6,z)-b\zeta(3,z)\Big)$ should be in the next two sections.

In the rest of this section, we shall prove Theorem \ref{Th112}.

By Theorem \ref{ThA}, for $b\in(3.06,3.062)$,
\begin{equation}\aligned\label{LLLa}
\min_{z\in\overline{\mathcal{D}_{\mathcal{G}}}\cap\{y\geq1\}}\Big(\zeta(6,z)-b\zeta(3,z)\Big)
=\min_{z\in\Gamma_a}\Big(\zeta(6,z)-b\zeta(3,z)\Big).
\endaligned\end{equation}
See Picture \ref{PFFF} for geometric shapes $\overline{\mathcal{D}_{\mathcal{G}}}\cap\{y\geq1\}$ and $\Gamma_a$.
On the other hand, by Theorem in the next section, one has
for $b\in(3.06,3.062)$,
\begin{equation}\aligned\label{LLLb}
\min_{z\in\Gamma_a}\Big(\zeta(6,z)-b\zeta(3,z)\Big)\;\;
\hbox{is achieved uniquely at }i.
\endaligned\end{equation}
By \eqref{LLLa} and \eqref{LLLb},
for $b\in(3.06,3.062)$,
\begin{equation}\aligned\label{LLLc}
\min_{z\in\overline{\mathcal{D}_{\mathcal{G}}}\cap\{y\geq1\}}\Big(\zeta(6,z)-b\zeta(3,z)\Big)\;\;
\hbox{is achieved at }i.
\endaligned\end{equation}
To prove Theorem \ref{Th113}, by \eqref{LLLc}, it suffices to prove the following lemma:

\begin{lemma}\label{Lemma111} For $b\in(3.06,3.062)$,
\begin{equation}\aligned\nonumber
\min_{z\in\overline{\mathcal{D}_{\mathcal{G}}}\cap\{y\in[\frac{\sqrt3}{2},1]\}}\Big(\zeta(6,z)-b\zeta(3,z)\Big)\;\;
\hbox{is achieved at }i.
\endaligned\end{equation}
\end{lemma}
To prove Lemma \ref{Lemma111}, we split the region $\overline{\mathcal{D}_{\mathcal{G}}}\cap\{y\in[\frac{\sqrt3}{2},1]\}$(see Picture \ref{PFFF}) into two subregions, namely,
 $$
 \overline{\mathcal{D}_{\mathcal{G}}}\cap\{y\in[\frac{\sqrt3}{2},1]\}=
 \Big(\overline{\mathcal{D}_{\mathcal{G}}}\cap\{y\in[\frac{\sqrt21}{5},1]\}\Big)
 \cup
 \Big(\overline{\mathcal{D}_{\mathcal{G}}}\cap\{y\in[\frac{\sqrt3}{2},\frac{\sqrt21}{5}]\}\Big)
 .
 $$
 In each subregion, we use different methods. In the first subregion, we have
\begin{lemma}\label{Lemma111a} For $b\in(3.06,3.062)$,
\begin{equation}\aligned\nonumber
\min_{z\in\overline{\mathcal{D}_{\mathcal{G}}}\cap\{y\in[\frac{\sqrt21}{5},1]\}}\Big(\zeta(6,z)-b\zeta(3,z)\Big)\;\;
\hbox{is achieved at }i.
\endaligned\end{equation}
\end{lemma}

A similar way to prove Theorem \ref{ThA} shows that
\begin{lemma}\label{Lemma1a} For $b\geq3.06$,
\begin{equation}\aligned\nonumber
\frac{\partial}{\partial x}\Big(\zeta(6,z)-b\zeta(3,z)\Big)\geq0, \;\; z\in\overline{\mathcal{D}_{\mathcal{G}}}\cap\{y\in[\frac{\sqrt21}{5},1]\}
.
\endaligned\end{equation}
\end{lemma}

A consequence of Lemma \ref{Lemma1a} gives that
\begin{lemma}\label{Lemma1b} For $b\geq3.06$,
\begin{equation}\aligned\nonumber
\min_{z\in\overline{\mathcal{D}_{\mathcal{G}}}\cap\{y\in[\frac{\sqrt21}{5},1]\}}\Big(\zeta(6,z)-b\zeta(3,z)\Big)
=\min_{z\in\Gamma_b\cap\{y\in[\frac{\sqrt21}{5},1]\}}\Big(\zeta(6,z)-b\zeta(3,z)\Big).
\endaligned\end{equation}
\end{lemma}
A case of Theorem \ref{Thbb} in  Section 7 shows that
\begin{lemma}\label{Lemma1c} For $b\geq3.06$,
\begin{equation}\aligned\nonumber
\min_{z\in\Gamma_b\cap\{y\in[\frac{\sqrt21}{5},1]\}}\Big(\zeta(6,z)-b\zeta(3,z)\Big)
\hbox{is achieved at }i.
\endaligned\end{equation}
\end{lemma}
Then Lemma \ref{Lemma111a} is prove by Lemmas \ref{Lemma1b} and \ref{Lemma1c}.
In second subregion $\overline{\mathcal{D}_{\mathcal{G}}}\cap\{y\in[\frac{\sqrt3}{2},\frac{\sqrt21}{5}]\}$, we show that
\begin{lemma}\label{Lemma111b} For $b\in(3.06,3.062)$,
\begin{equation}\aligned\nonumber
\Big(\zeta(6,z)-b\zeta(3,z)\Big)-\Big(\zeta(6,z)-b\zeta(3,z)\Big)\mid_{z=i}\geq10^{-3},\;\;
\hbox{for}\;\;z\in\overline{\mathcal{D}_{\mathcal{G}}}\cap\{y\in[\frac{\sqrt3}{2},\frac{\sqrt21}{5}]\}.
\endaligned\end{equation}
\end{lemma}

In fact, the subregion $\overline{\mathcal{D}_{\mathcal{G}}}\cap\{y\in[\frac{\sqrt3}{2},\frac{\sqrt21}{5}]\}$ is very small(see Picture \ref{PFFF}) can be split into $2^8$ sub-subregions, Lemma \ref{Lemma111b} is showed by direct checking. The checking is accurate since there is a uniform positive lower bound.

Lemma \ref{Lemma111} is proved by Lemmas \ref{Lemma111a} and \ref{Lemma111b}.
The proof of Theorem \ref{Th112} is complete.

\section{Numbers and locations of critical points of $\Big(\zeta(6,z)-b\zeta(3,z)\Big)$ when $z\in\Gamma_a$}\label{section6}

Recall that
\begin{equation}\aligned\nonumber
\Gamma_a&=\{
z\in\mathbb{H}: \Re(z)=0,\; \Im(z)\geq1
\};\\
\Gamma_b&=\{
z\in\mathbb{H}: |z|=1,\; \Im(z)\in[\frac{\sqrt3}{2},1]
\}.
\endaligned\end{equation}
And
\begin{equation}\aligned\nonumber
\Gamma=\Gamma_a\cup\Gamma_b.
\endaligned\end{equation}

By the main theorems in previous sections, we have reduced to the minimizers of the functional $\Big(\zeta(6,z)-b\zeta(3,z)\Big)$
to the curve $\Gamma$ for all $b$. Namely, we have
\begin{theorem}[Theorems \ref{ThA}, \ref{Th102} and \ref{Th112}]
For $b\in\R$,
\begin{equation}\aligned\nonumber
\min_{z\in\mathbb{H}}\Big(\zeta(6,z)-b\zeta(3,z)\Big)=\min_{z\in\overline{\mathcal{D}_{\mathcal{G}}}}\Big(\zeta(6,z)-b\zeta(3,z)\Big)
=\min_{z\in\Gamma}\Big(\zeta(6,z)-b\zeta(3,z)\Big)
.
\endaligned\end{equation}
\end{theorem}
We then analyze the functional $\Big(\zeta(6,z)-b\zeta(3,z)\Big)$ with parameter $b$ on $\Gamma_a$ and $\Gamma_b$ in this and next section respectively.

The following is main result of this section:
\begin{theorem}[{\bf Location of minimizers on $\Gamma_a$}]\label{Thaa}
The minimizers of $\Big(
\zeta(6,z)-b\zeta(3,z)
\Big)$ on $\Gamma_a$ for $b\in\R$ are given by
  \begin{equation}\aligned\nonumber
\min_{z\in\Gamma_a}\Big(
\zeta(6,z)-b\zeta(3,z)
\Big)\;\;\hbox{is achieved at}\;\;\begin{cases}
i;\;\;\;\;  b\leq b_a;\\
iy_b;\;\; y_b>1;\;\;  b>b_a.
\end{cases}
\endaligned\end{equation}
The threshold $b_a$ and parameter $y_b$ are determined by
\begin{equation}\aligned\nonumber
b_{a}=\frac{\frac{\partial\zeta(6,iy)}{\partial y}}
{\frac{\partial\zeta(3,iy)}{\partial y}}\mid_{y=1}=4.0774\cdots.
\endaligned\end{equation}
\begin{equation}\aligned\nonumber
  \frac{dy_b}{db}>0,\;\;\;\;\lim_{b\rightarrow\infty}\frac{y_b}{\sqrt[3]{\frac{\xi(6)}{2\xi(12)}\cdot b}}=1.
  \endaligned\end{equation}

\end{theorem}

Note that $z\in\Gamma_a$ implies that $z=iy$. Theorem \ref{Thaa} is implied by the following Proposition \ref{Prop1},
 which describes the global picture of the critical points
 of $\Big(\zeta(6,iy)-b\zeta(3,iy)\Big), b\in \R, y\in(0,\infty)$ for various parameter $b$.
\begin{proposition}[{\bf Number and location of critical points of $\Big(\zeta(6,iy)-b\zeta(3,iy)\Big)$}]\label{Prop1}
The function $\Big(\zeta(6,iy)-b\zeta(3,iy)\Big), b\in \R, y\in(0,\infty)$ has one or three critical points depending on the value of $b$.
Define that
\begin{equation}\aligned\label{ba}
b_{a}=\frac{\frac{\partial\zeta(6,iy)}{\partial y}}
{\frac{\partial\zeta(3,iy)}{\partial y}}\mid_{y=1}=4.0774\cdots.
\endaligned\end{equation}
\begin{itemize}
  \item [(1)] if $b\leq b_a$, $\Big(\zeta(6,iy)-b\zeta(3,iy)\Big)$ has only one critical point $y=1$, and is increasing on $(1,\infty)$;
  \item [(2)] if $b> b_a$, $\Big(\zeta(6,iy)-b\zeta(3,iy)\Big)$ has three critical points, $1, y_b(y_b>1), \frac{1}{y_b}$, is decreasing on $(1,y_b)$ and increasing on $(y_b,\infty)$.
      $\lim_{b\rightarrow\infty}y_b=\infty$, further
  \begin{equation}\aligned\label{yb}
  \frac{dy_b}{db}>0,\;\;\;\;\lim_{b\rightarrow\infty}\frac{y_b}{\sqrt[3]{\frac{\xi(6)}{2\xi(12)}\cdot b}}=1.
  \endaligned\end{equation}

\end{itemize}
\end{proposition}
To prove Proposition \ref{Prop1}, we use the deformation:
\begin{equation}\aligned\label{K1}
\frac{\partial }{\partial y}\Big(
\zeta(6,iy)-b\zeta(3,iy)
\Big)
=\frac{\partial }{\partial y}\zeta(3,iy)\cdot
\Big(
\frac{\frac{\partial\zeta(6,iy)}{\partial y}}
{\frac{\partial\zeta(3,iy)}{\partial y}}-b
\Big).
\endaligned\end{equation}
On the other hand, it is known that(see e.g. \cite{Ran1953,Cas1959})
\begin{equation}\aligned\nonumber
\frac{\partial }{\partial y}\zeta(3,iy)
\begin{cases}
>0, \;\;y>1;\\
=0,\;\;y=1;\\
<0,\;\; y\in(0,1).
\end{cases}
\endaligned\end{equation}
Then by \eqref{K1}, the proof of Proposition \ref{Prop1} is reduced to the following proposition \ref{Prop2}.

\begin{proposition}
[{\bf The shape of $\frac{\frac{\partial\zeta(6,iy)}{\partial y}}
{\frac{\partial\zeta(3,iy)}{\partial y}}$}]\label{Prop2} The function
$\frac{\frac{\partial\zeta(6,iy)}{\partial y}}
{\frac{\partial\zeta(3,iy)}{\partial y}}$ on $y\in(0,\infty)$ admits only one critical point $y=1$ on $(0,\infty)$. Further,
\begin{equation}\aligned\nonumber
\frac{\frac{\partial\zeta(6,iy)}{\partial y}}
{\frac{\partial\zeta(3,iy)}{\partial y}} \;\;\hbox{is decreasing on}\;\;(0,1)\;\;\hbox{and}\;\;\hbox{is increasing on}\;\;(1,\infty).
\endaligned\end{equation}

\end{proposition}

The proof of Proposition \ref{Prop2} is split into several smaller lemmas in the following. Before introducing these smaller lemmas,
we shall introduce some notations.

Define
\begin{equation}\aligned\label{XY}
\mathcal{X}(y):&=\frac{\frac{\partial\zeta(6,iy)}{\partial y}}
{\frac{\partial\zeta(3,iy)}{\partial y}}, \;\;\mathcal{X}'(y)=\frac{\partial}{\partial y}\frac{\frac{\partial\zeta(6,iy)}{\partial y}}
{\frac{\partial\zeta(3,iy)}{\partial y}},y>0,\\
\mathcal{Y}(y):&=\frac{\partial^2\zeta(6,iy)}{\partial y^2}\frac{\partial\zeta(3,iy)}{\partial y}
-\frac{\partial\zeta(6,iy)}{\partial y}\frac{\partial^2\zeta(3,iy)}{\partial y^2}, y>0.
\endaligned\end{equation}
Then
\begin{equation}\aligned\label{XY1}
\mathcal{Y}(y)=\mathcal{X}'(y)\cdot \big(\frac{\partial\zeta(3,iy)}{\partial y}\big)^2, y>0.
\endaligned\end{equation}
And
\begin{equation}\aligned\nonumber
\mathcal{Y}'(y)=\frac{\partial^3\zeta(6,iy)}{\partial y^3}\frac{\partial\zeta(3,iy)}{\partial y}
-\frac{\partial\zeta(6,iy)}{\partial y}\frac{\partial^3\zeta(3,iy)}{\partial y^3}, y>0.
\endaligned\end{equation}

One first has the following lemma
\begin{lemma}\label{Lemmag1}
\begin{equation}\aligned\nonumber
\mathcal{X}(\frac{1}{y})=\mathcal{X}(y).
\endaligned\end{equation}
Consequently,
\begin{equation}\aligned\nonumber
\mathcal{X}'(\frac{1}{y})=-y^2\cdot\mathcal{X}'(y),\;\; \mathcal{X}'(1)=0.
\endaligned\end{equation}
\end{lemma}
The proof of Lemma \ref{Lemmag1} is based on
\begin{lemma}\label{Lemmag2} For any $s>0$, $\zeta(s,iy)$ satisfies the functional equation:
\begin{equation}\aligned\nonumber
\mathcal{F}(\frac{1}{y})=\mathcal{F}(y).
\endaligned\end{equation}
\end{lemma}
Which follows from the invariance of $\zeta (s,z)$ by the transformation $z \to -\frac{1}{z}$. See Lemma \ref{FRegion}.

Using Lemma \ref{Lemmag2} twice and taking derivatives, one gets Lemma \ref{Lemmag1}. From Lemmas \ref{Lemmag1}-\ref{Lemmag2} we deduce that
\begin{lemma}
\label{LemmaD1}
\begin{equation}\aligned\nonumber
\Big(\frac{\partial^3\zeta(6,iy)}{\partial y^3}\frac{\partial^2\zeta(3,iy)}{\partial y^2}
-\frac{\partial^2\zeta(6,iy)}{\partial y^2}\frac{\partial^3\zeta(3,iy)}{\partial y^3}\Big)
\mid_{y=1}=0.
\endaligned\end{equation}
\end{lemma}
\begin{proof} Taking second and third order derivative on functional equation in Lemma \ref{Lemmag2}, one gets
\begin{equation}\aligned\label{Fabc}
\mathcal{F}''(\frac{1}{y})&=2y^3\mathcal{F}'(y)+y^4\mathcal{F}''(y)\\
\mathcal{F}'''(\frac{1}{y})&=-6y^4\mathcal{F}'(y)-6y^5\mathcal{F}''(y)-y^6\mathcal{F}'''(y).
\endaligned\end{equation}
Evaluating $y=1$ in \eqref{Fabc}, one has
\begin{equation}\aligned\nonumber
\mathcal{F}'''(1)=-3\mathcal{F}''(1).
\endaligned\end{equation}
Then by Lemma \ref{Lemmag2},
\begin{equation}\aligned\label{Fabc2}
\frac{\partial^3\zeta(6,iy)}{\partial y^3}=-3\frac{\partial^2\zeta(6,iy)}{\partial y^2},\;\;
\frac{\partial^3\zeta(3,iy)}{\partial y^3}=-3\frac{\partial^2\zeta(3,iy)}{\partial y^2}.
\endaligned\end{equation}
\eqref{Fabc2} gives the result.
\end{proof}
\begin{lemma}\label{LemmaY123}
\begin{equation}\aligned
\mathcal{Y}(1)=\mathcal{Y}'(1)=\mathcal{Y}''(1)=0.
\endaligned\end{equation}
\end{lemma}
\begin{proof}
The part $\mathcal{Y}(1)=\mathcal{Y}'(1)=0$ follows directly by
\begin{equation}\aligned\label{YYY12}
\frac{\partial\zeta(6,iy)}{\partial y}\mid_{y=1}=0,\;\;
\frac{\partial\zeta(3,iy)}{\partial y}\mid_{y=1}=0
\endaligned\end{equation}
deduced by Lemma \ref{Lemmag2}. The part $\mathcal{Y}''(1)=0$ follows by \eqref{YYY12} and Lemma \ref{LemmaD1}.
\end{proof}

By Lemma \ref{Lemmag1}, to prove Proposition \ref{Prop2}, it suffices to prove that

\begin{lemma}\label{Lemmag3}
\begin{equation}\aligned\nonumber
\mathcal{Y}(y)>0,\;\hbox{for}\;\; y>1.
\endaligned\end{equation}
Here the function $\mathcal{Y}$ is defined in \eqref{XY}.
\end{lemma}

Since by Lemma \ref{LemmaY123}, $\mathcal{Y}(1)=0$.
Therefore to prove Lemma \ref{Lemmag3}, we divide it into two cases, namely, near and away from $y=1$, i.e., $y\in(1,1.05]$ and $y\in[1.05,\infty)$.
For the case away from $y=1$, namely, the case $y\geq1.05$, we estimate directly by
\begin{lemma}\label{Lemmag54}It holds that
\begin{equation}\aligned\nonumber
\mathcal{Y}(y)\geq4.5>0,\;\hbox{for}\;\; y\in[1.05,\infty).
\endaligned\end{equation}
\end{lemma}
For the case $y\in(1,1.05]$, one has by direct computation
\begin{lemma}\label{Lemmag4}
\begin{equation}\aligned
\mathcal{Y}'''(y)\geq10^5>0,\;\;\hbox{for}\;\;y\in[1,1.05],
\endaligned\end{equation}
and hence
\begin{equation}\aligned\nonumber
\mathcal{Y}(y)>0,\;\hbox{for}\;\; y\in[1,1.05].
\endaligned\end{equation}
\end{lemma}

\section{Number and location of critical points of $\Big(\zeta(6,z)-b\zeta(3,z)\Big)$ when $z\in\Gamma_b$}\label{section7}

Recall that the definition of the $\frac{1}{4}-$arc:
\begin{equation}\aligned\nonumber
\Gamma_b=\{
z\in\mathbb{H}: |z|=1,\; y\in[\frac{\sqrt3}{2},1]
\}.
\endaligned\end{equation}

In this section, we aim to characterize the minimizers of $\Big(\zeta(6,z)-b\zeta(3,z)\Big)$ on $\Gamma_b$ for all $b\in\R$.
It is stated in the following theorem.

\begin{theorem}[{\bf Location of the minimizers on $\Gamma_b$}]\label{Thbb}
For the minimizers of $\Big(\zeta(6,z)-b\zeta(3,z)\Big)$ on $\Gamma_b$ for $b\in\R$,
there exists thresholds $b_1,b_2$ defined by
\begin{equation}\aligned\nonumber
b_1:=2.9463\cdots,\;\;
b_2:=\frac{\frac{\partial \zeta(6,\frac{1}{2}+iy)}{\partial y}}{\frac{\partial\zeta(3,\frac{1}{2}+iy)}{\partial y}}|_{y=\frac{1}{2}}
=2.9866\cdots
\endaligned\end{equation}
such that
\begin{equation}\aligned\nonumber
\min_{z\in\Gamma_b}\Big(\zeta(6,z)-b\zeta(3,z)\Big)\;\;\hbox{is achieved at}\;\;\begin{cases}
e^{i\frac{\pi}{3}},\;\;\;\;\;\;\;\; \;\;\;\;b\in (-\infty,b_1);\\
e^{i\frac{\pi}{3}}, \;\;\hbox{or}\;\;e^{i\theta_{b_1}},\;\; b=b_1;\\
e^{i\theta_{b}},\;\;\;\;\;\;\;\;\;\;\;\; b\in (b_1,b_2);\\
i,\;\;\;\;\;\;\;\;\;\;\;\;\;\;\;\; b\in[b_2,\infty).
\end{cases}
\endaligned\end{equation}
Here the argument $\theta_{b_1}$ associated to the threshold $b_1$ is numerically computed by
\begin{equation}\aligned\nonumber
\theta_{b_1}:=1.33473846\cdots.
\endaligned\end{equation}
And for the argument $\theta_{b}$ associated to the parameter $b$, it holds that
  \begin{equation}\aligned\nonumber
\frac{d}{db}\theta_{b}>0\;\;\hbox{and}\;\;\theta_b\in(\theta_{b_1},\frac{\pi}{2})\;\;\hbox{for}\;b\in(b_1,b_2).
\endaligned\end{equation}
\end{theorem}
Note that there are two thresholds exist in locating the minimizers of $\Big(\zeta(6,z)-b\zeta(3,z)\Big)$ on $\Gamma_b$. Besides,
when the parameter $b$ touches the first parameter $b_1$, there are two different minimizers.

To study to the minimizers of $\Big(\zeta(6,z)-b\zeta(3,z)\Big)$ on $\Gamma_b$, we first show that it is equivalent to studying
the minimizers on a straight vertical line. We state it in the following lemma:
\begin{lemma}[{\bf From the arc $\Gamma_b$ to the $\frac{1}{2}-$axis}]\label{LemmaH1}
\begin{equation}\aligned\nonumber
&\Big(\zeta(6,u+i\sqrt{1-u^2})-b\zeta(3,u+i\sqrt{1-u^2})\Big)\\
=&\Big(\zeta(6,\frac{1}{2}+\frac{i}{2}\sqrt{\frac{1+u}{1-u}})-b\zeta(3,\frac{1}{2}+\frac{i}{2}\sqrt{\frac{1+u}{1-u}})\Big)
\endaligned\end{equation}
\end{lemma}
Lemma \ref{LemmaH1} is induced by the fact that $\zeta(s,z), s>0$ is variant by the transform $z\mapsto\frac{1}{1-z}$.
Note that $z\in\Gamma_b\Leftrightarrow z=u+i\sqrt{1-u^2}, u\in[0,\frac{1}{2}]$. Then by Lemma \ref{LemmaH1}, to study
$\Big(\zeta(6,z)-b\zeta(3,z)\Big)$ on $\Gamma_b$, it is equivalent to studying $\Big(\zeta(6,\frac{1}{2}+iy)-b\zeta(3,\frac{1}{2}+iy)\Big)$ for  $y\in[\frac{1}{2},\frac{\sqrt3}{2}]$.
By Lemma \ref{LemmaH1}, to prove Theorem \ref{Thbb}, it equivalents to prove that
\begin{theorem}[{\bf Location of the minimizers on short interval of $\frac{1}{2}-$axis}]\label{Thbb2}
Let the critical values $b_1,b_2$ be defined by
\begin{equation}\aligned\nonumber
b_1:=2.9463\cdots,\;\;
b_2:=\frac{\frac{\partial \zeta(6,\frac{1}{2}+iy)}{\partial y}}{\frac{\partial\zeta(3,\frac{1}{2}+iy)}{\partial y}}|_{y=\frac{1}{2}}
=2.9866\cdots.
\endaligned\end{equation}
Then
\begin{equation}\aligned\nonumber
\min_{y\in[\frac{1}{2},\frac{\sqrt3}{2}]}\Big(\zeta(6,\frac{1}{2}+iy)-b\zeta(3,\frac{1}{2}+iy)\Big)\;\;\hbox{is achieved at}\;\;\begin{cases}
\frac{\sqrt3}{2},\;\;\;\;\;\;\;\; \;\;\;\;b\in (-\infty,b_1);\\
\frac{\sqrt3}{2}, \;\;\hbox{or}\;\;y_{b_1},\;\; b=b_1;\\
y_b,\;\;\;\;\;\;\;\;\;\;\;\; b\in (b_1,b_2);\\
\frac{1}{2},\;\;\;\;\;\;\;\;\;\;\;\;\;\;\;\; b\in[b_2,\infty).
\end{cases}
\endaligned\end{equation}
Here the minimizer $y_{b_1}$ is numerically computed by
\begin{equation}\aligned\nonumber
y_{b_1}:=0.6346835\cdots.
\endaligned\end{equation}
And for the minimizer $y_{b}$ associated to the parameter $b$, it holds that
  \begin{equation}\aligned\nonumber
\frac{d}{db}y_{b}<0\;\;\hbox{and}\;\;y_b\in(\frac{1}{2},y_{b_1})\;\;\hbox{for}\;b\in(b_1,b_2).
\endaligned\end{equation}
\end{theorem}

To prove Theorem \ref{Thbb2}, we shall study the critical points of $\Big(\zeta(6,\frac{1}{2}+iy)-b\zeta(3,\frac{1}{2}+iy)\Big)$ for $y\in[\frac{1}{2},\frac{\sqrt3}{2}]$. Before this, one has the universal critical points for all $b\in\R$. Namely,
\begin{lemma}[]\label{LemmaH2} For all $b\in\R$
\begin{equation}\aligned\nonumber
\Big(\zeta(6,\frac{1}{2}+iy)-b\zeta(3,\frac{1}{2}+iy)\Big)\mid_{y\in[\frac{1}{2},\frac{\sqrt3}{2}]}\;\;\hbox{admits two cirical points at}\;\;y=\frac{1}{2}\;\;\hbox{and}\;\;y=\frac{\sqrt3}{2}.
\endaligned\end{equation}
\end{lemma}

Lemma \ref{LemmaH2} dues to that $\frac{1}{2},\frac{\sqrt3}{2}$ are critical points of $\zeta(s,iy), s>0$[See e.g. \cite{Ran1953}].

In the next, we characterize the critical points of $\Big(\zeta(6,\frac{1}{2}+iy)-b\zeta(3,\frac{1}{2}+iy)\Big)\mid_{y\in[\frac{1}{2},\frac{\sqrt3}{2}]}$ for various $b$.

\begin{proposition}[{\bf Number and location of critical points of $\Big(\zeta(6,\frac{1}{2}+iy)-b\zeta(3,\frac{1}{2}+iy)\Big)\mid_{y\in[\frac{1}{2},\frac{\sqrt3}{2}]}$}]\label{Prop71} The function
$\Big(\zeta(6,\frac{1}{2}+iy)-b\zeta(3,\frac{1}{2}+iy)\Big)\mid_{y\in[\frac{1}{2},\frac{\sqrt3}{2}]}$ admits 2 or 3 or 4 critical points.
To classify the critical points, we introduce the critical parameters in order:
\begin{equation}\aligned\nonumber
b_{0}:&=\min_{y\in[\frac{1}{2},\frac{\sqrt3}{2}]}\frac{\frac{\partial \zeta(6,\frac{1}{2}+iy)}{\partial y}}{\frac{\partial\zeta(3,\frac{1}{2}+iy)}{\partial y}}=2.9322\cdots,\\
y_0:&=\;\;\hbox{the unique minimizer of}\min_{y\in[\frac{1}{2},\frac{\sqrt3}{2}]}\frac{\frac{\partial \zeta(6,\frac{1}{2}+iy)}{\partial y}}{\frac{\partial\zeta(3,\frac{1}{2}+iy)}{\partial y}}=0.7035146\cdots,\;\;\\
b_{2}:&=\frac{\frac{\partial \zeta(6,\frac{1}{2}+iy)}{\partial y}}{\frac{\partial\zeta(3,\frac{1}{2}+iy)}{\partial y}}|_{y=\frac{1}{2}}
=2.9866\cdots,\\
b_{3}:&=\frac{\frac{\partial \zeta(6,\frac{1}{2}+iy)}{\partial y}}{\frac{\partial\zeta(3,\frac{1}{2}+iy)}{\partial y}}|_{y=\frac{\sqrt3}{2}}
=3.0614\cdots.
\endaligned\end{equation}

It is classified to five cases as follows:
\begin{itemize}
  \item [(1)] $b\in(-\infty, b_{0})$, there are only two critical points, i.e., $\{\frac{1}{2}, \frac{\sqrt3}{2}\}$, where $\frac{1}{2}$ is a local maxima, $\frac{\sqrt3}{2}$ is a local minima;
  \item [(2)] $b=b_0$, it admits exactly three critical points, i.e., $\{\frac{1}{2}, y_0, \frac{\sqrt3}{2}\}$, where $\frac{1}{2}$ is a local maxima, $\frac{\sqrt3}{2}$ is a local minima and $y_0$ is a saddle point;
  \item [(3)] $b\in(b_{0},b_{2})$, there are exactly four critical points, denoted by $\{\frac{1}{2},y_b, y_{b'}, \frac{\sqrt3}{2}\}$, where $\frac{1}{2}$ is a local maxima, $\frac{\sqrt3}{2}$ is a local minima; $y_b$ is a local minima and $y_{b'}$ is a local maxima, $y_b\in(\frac{1}{2},y_0)$, $y_{b'}\in(y_0,\frac{\sqrt3}{2})$;
further $y_b$ is decreasing with $b$, $y_{b'}$ is increasing with $b$, i.e.,

   \begin{equation}\aligned\nonumber
\frac{d}{db}y_b<0, \;\; \frac{d}{db}y_{b'}>0.
\endaligned\end{equation}

  \item [(4)] $b\in[b_{2},b_{3})$, there are exactly three critical points, i.e., $\{\frac{1}{2}, y_{b''}, \frac{\sqrt3}{2}\}$, $\frac{1}{2}$ and $\frac{\sqrt3}{2}$ are local minima; $y_{b''}$ is a local maxima;
  \item [(5)] $b\in[b_{3},\infty)$, there are only two critical points, i.e., $\{\frac{1}{2}, \frac{\sqrt3}{2}\}$, $\frac{1}{2}$ is a local minima and $\frac{\sqrt3}{2}$ is a local maxima.
\end{itemize}
\end{proposition}
We shall postpone the proof of Proposition \ref{Prop71} to the later and state a direct consequence of Proposition \ref{Prop71}, which can partially prove Theorem \ref{Thbb2}. We state in the following corollary:

\begin{corollary}\label{CoroA}
\begin{equation}\aligned\nonumber
\min_{y\in[\frac{1}{2},\frac{\sqrt3}{2}]}\Big(\zeta(6,\frac{1}{2}+iy)-b\zeta(3,\frac{1}{2}+iy)\Big)\;\;\hbox{is achieved at}\;\;\begin{cases}
\frac{\sqrt3}{2},\;\;\;\;\;\;\;\; \;\;\;\;b\in (-\infty,b_0];\\
\frac{\sqrt3}{2} \;\;\hbox{or}\;\;y_{b},\;\; b\in(b_0, b_2);\\
\frac{\sqrt3}{2}\;\;\hbox{or}\;\;\frac{1}{2}\;\;\;\; b\in [b_2,b_3);\\
\frac{1}{2},\;\;\;\;\;\;\;\;\;\;\;\;\;\;\;\; b\in[b_3,\infty).
\end{cases}
\endaligned\end{equation}
Here $y_b$ is defined in Proposition \ref{Prop71}.
\end{corollary}

By Corollary \ref{CoroA}, to prove Theorem \ref{Thbb2}, it suffices to further locate precisely the minimizer of
$\min_{y\in[\frac{1}{2},\frac{\sqrt3}{2}]}\Big(\zeta(6,\frac{1}{2}+iy)-b\zeta(3,\frac{1}{2}+iy)\Big)$ for $b\in(b_0,b_3)$.
For this, one has the following preliminary lemma:

\begin{lemma}[A comparison between the values at $\frac{\sqrt3}{2}$ and $\frac{1}{2}$ ]\label{LemmaH3} Let
\begin{equation}\aligned\nonumber
b_{1.5}:=\frac{\zeta(6,\frac{1}{2}+i\frac{1}{2})-\zeta(6,\frac{1}{2}+i\frac{\sqrt3}{2})}
{\zeta(3,\frac{1}{2}+i\frac{1}{2})-\zeta(3,\frac{1}{2}+i\frac{\sqrt3}{2})}=2.9529\cdots.
\endaligned\end{equation}
Then
\begin{equation}\aligned\nonumber
&\Big(\zeta(6,\frac{1}{2}+iy)-b\zeta(3,\frac{1}{2}+iy)\Big)\mid_{y=\frac{\sqrt3}{2}}
\geq\Big(\zeta(6,\frac{1}{2}+iy)-b\zeta(3,\frac{1}{2}+iy)\Big)\mid_{y=\frac{1}{2}}\\
\Leftrightarrow &b\geq b_{1.5}.
\endaligned\end{equation}

\end{lemma}
\begin{remark}
We use the notation $b_{1.5}$ to denote the intermediate parameter instead of $b_1$ and use $b_1$ to denote the first critical parameter to be introduced later.

\end{remark}
Since $b_2>b_{1.5}$, a direct consequence of Lemma \ref{LemmaH3} and Proposition \ref{Prop71} yield that

\begin{corollary}\label{CoroB}
\begin{equation}\aligned\nonumber
\min_{y\in[\frac{1}{2},\frac{\sqrt3}{2}]}\Big(\zeta(6,\frac{1}{2}+iy)-b\zeta(3,\frac{1}{2}+iy)\Big)\;\;\hbox{is achieved at}\;\;\begin{cases}
y_b, \;\;\;\;\;\;\;\;\;\; b\in[b_{1.5}, b_2);\\
\frac{1}{2},\;\;\;\;\;\;\;\;\;\;\; b\in [b_2,b_3);\\
\end{cases}
\endaligned\end{equation}
Here $y_b$ is defined in Proposition \ref{Prop71}.
\end{corollary}

By Corollaries \ref{CoroA} and \ref{CoroB}, to prove Theorem \ref{Thbb2}, it needs to
locate precisely the minimizer of
$\min_{y\in[\frac{1}{2},\frac{\sqrt3}{2}]}\Big(\zeta(6,\frac{1}{2}+iy)-b\zeta(3,\frac{1}{2}+iy)\Big)$ for $b\in(b_0,b_{1.5})$.
In fact, by Corollary \ref{CoroA}, one has
\begin{equation}\aligned\nonumber
\min_{y\in[\frac{1}{2},\frac{\sqrt3}{2}]}\Big(\zeta(6,\frac{1}{2}+iy)-b\zeta(3,\frac{1}{2}+iy)\Big)\;\;\hbox{is achieved at}\;\;
y_b\;\;\;\hbox{or}\;\;\frac{\sqrt3}{2}, \;\;\;\hbox{for}\;\; b\in(b_0,b_{1.5}).
\endaligned\end{equation}
We shall further decide precisely where the minimizer is $\frac{\sqrt3}{2}$ and where minimizer is $y_b$(see Picture \ref{Transition} for a visible image).
To this, we introduce the following function, which is the difference of functional evaluating at $y_b$ and $\frac{\sqrt3}{2}$
\begin{equation}\aligned\label{fbbb}
{g}(b,y):&=\Big(\zeta(6,\frac{1}{2}+iy)-b\zeta(3,\frac{1}{2}+iy)\Big), b\in(b_0,b_{1.5})\\
f(b):&={g}(b,y_b)-{g}(b,\frac{\sqrt3}{2}), b\in(b_0,b_{1.5}).
\endaligned\end{equation}
Note that when $b\in(b_0,b_{1.5})$, both $\frac{\sqrt3}{2}$ and $y_b$ are critical points of $\Big(\zeta(6,\frac{1}{2}+iy)-b\zeta(3,\frac{1}{2}+iy)\Big)$ by Proposition \ref{Prop71}.
We shall investigate the properties of $f(b)$ defined in \eqref{fbbb} as follows:

\begin{lemma}[{\bf A comparison between the values at $\frac{\sqrt3}{2}$ and $y_b$}]\label{LemmaH4}
$f(b)$ has the following simple structure:
\begin{itemize}
  \item [(1)] $f(b_{0})>0$;
  \item [(2)] $f(b_{1.5})<0$;
  \item [(3)] $\frac{d}{db}f(b)<0$;
  \item [(4)] $f(b)$ admits only one root denoted by $b_1$ in $(b_0, b_{1.5})$. Numerically,
 \begin{equation}\aligned\nonumber
b_1=2.9463\cdots,
\endaligned\end{equation}
and
 \begin{equation}\aligned\nonumber
y_{b_1}=0.6346835\cdots.
\endaligned\end{equation}
\end{itemize}

\end{lemma}
\begin{figure}
\centering
 \includegraphics[scale=0.48]{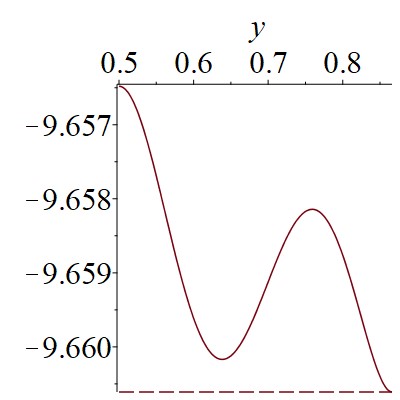}\includegraphics[scale=0.48]{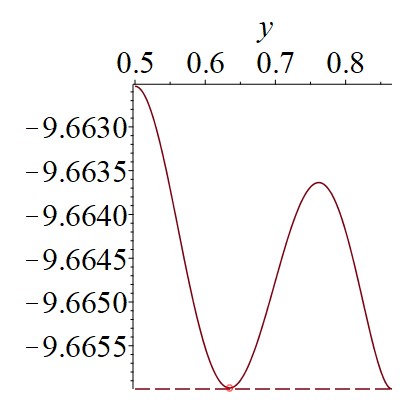}\includegraphics[scale=0.48]{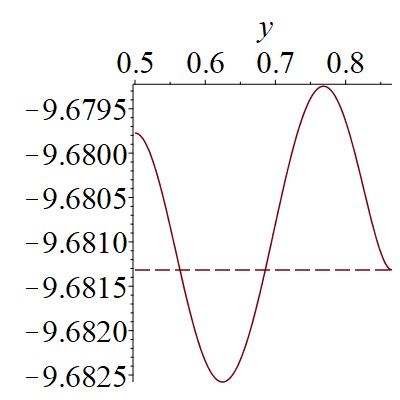}
 \caption{The transition of minimizer of $\Big(\zeta(6,\frac{1}{2}+iy)-b\zeta(3,\frac{1}{2}+iy)\Big)$
  for $b\in(b_0,b_{1.5})$ and location of $y_{b_1}$.
 }
\label{Transition}
\end{figure}

\begin{proof}[ Proof of Lemma \ref{LemmaH4}] Items $(1),(2)$ are proved by direct computation.
For Item $(3)$, we compute by \eqref{fbbb}
\begin{equation}\aligned\nonumber
f'(b)&=\frac{\partial y_b}{\partial b}\cdot
\Big(
\frac{\partial}{\partial y}
\Big(\zeta(6,\frac{1}{2}+iy)-b\zeta(3,\frac{1}{2}+iy)\Big)\mid_{y=y_b}\Big)
+\Big(\zeta(6,\frac{1}{2}+i\frac{\sqrt3}{2})-b\zeta(3,\frac{1}{2}+iy_b)\Big)\\
&=\Big(\zeta(6,\frac{1}{2}+i\frac{\sqrt3}{2})-\zeta(3,\frac{1}{2}+iy_b)\Big)\\
&<0.
\endaligned\end{equation}
Here one uses that $y_b$ is the critical point of $\Big(\zeta(6,\frac{1}{2}+iy)-b\zeta(3,\frac{1}{2}+iy)\Big)$ followed by Proposition \ref{Prop71} and $\frac{1}{2}+i\frac{\sqrt3}{2}$ is the global minimum of $\zeta(s,z), s>0$(\cite{Ran1953,Cas1959}).
The existence and uniqueness of the root $f(b)=0$ follows by Item $(3)$.
The root is denoted by $b_1$ and $b_1=2.9463\cdots$ numerically, consequently, $y_{b_1}=0.6346835\cdots$. This completes the proof.

\end{proof}

By Lemma \ref{LemmaH4}, we have
\begin{corollary}\label{CoroC}
\begin{equation}\aligned\nonumber
\min_{y\in[\frac{1}{2},\frac{\sqrt3}{2}]}\Big(\zeta(6,\frac{1}{2}+iy)-b\zeta(3,\frac{1}{2}+iy)\Big)\;\;\hbox{is achieved at}\;\;\begin{cases}
\frac{\sqrt3}{2}, \;\;\;\;\;\;\;\;\;\; b\in(b_0, b_1);\\
\frac{\sqrt3}{2}\;\;\hbox{or}\;\;y_{b_1},\;\;\;\;\;\;\;\;\; b=b_1;\\
y_b,\;\;\;\;\;\;\;\;\;\;\;\;b\in(b_1, b_{1.5}).
\end{cases}
\endaligned\end{equation}
Here $y_b$ and $b_1$ are defined in Proposition \ref{Prop71} and Lemma \ref{LemmaH4} respectively.
\end{corollary}

Note that Corollaries \ref{CoroA}, \ref{CoroB} and \ref{CoroC} complete the proof of Theorem \ref{Thbb2} and hence
the proof of Theorem \ref{Thbb}. It remains to prove Proposition \ref{Prop71}. For this, we use the following deformation:
\begin{equation}\aligned\label{J1}
\frac{\partial}{\partial y}\Big(\zeta(6,\frac{1}{2}+iy)-b\zeta(3,\frac{1}{2}+iy)\Big)
=\frac{\partial}{\partial y}\zeta(3,\frac{1}{2}+iy)
\cdot
\Big(
\frac{\frac{\partial}{\partial y}\zeta(6,\frac{1}{2}+iy)}{\frac{\partial}{\partial y}\zeta(3,\frac{1}{2}+iy)}
-b
\Big).
\endaligned\end{equation}
For $\frac{\partial}{\partial y}\zeta(3,\frac{1}{2}+iy)$, it is known that(see e.g. \cite{Ran1953,Cas1959})
\begin{equation}\aligned\label{J1b}
\frac{\partial }{\partial y}\zeta(3,\frac{1}{2}+iy)
\begin{cases}
<0, \;\;y\in(\frac{1}{2},\frac{\sqrt3}{2});\\
=0,\;\;y\in\{\frac{1}{2}, \frac{\sqrt3}{2}\}.
\end{cases}
\endaligned\end{equation}
By \eqref{J1} and \eqref{J1b}, to prove Proposition \ref{Prop71}(study the critical points of $\Big(\zeta(6,\frac{1}{2}+iy)-b\zeta(3,\frac{1}{2}+iy)\Big)$), it suffices to solve the equation
\begin{equation}\aligned\label{J2}
\frac{\frac{\partial}{\partial y}\zeta(6,\frac{1}{2}+iy)}{\frac{\partial}{\partial y}\zeta(3,\frac{1}{2}+iy)}
-b
=0,\;\;\hbox{where}\;\;y\in[\frac{1}{2}, \frac{\sqrt3}{2}].
\endaligned\end{equation}

We shall study the property of $\frac{\frac{\partial}{\partial y}\zeta(6,\frac{1}{2}+iy)}{\frac{\partial}{\partial y}\zeta(3,\frac{1}{2}+iy)}$
in the following proposition, it turns out that the function looks like a quadratic function on $[\frac{1}{2},\frac{\sqrt3}{2}]$.
\begin{proposition}\label{Prop72}
$(${\bf The shape of $\frac{\frac{\partial \zeta(6,\frac{1}{2}+iy)}{\partial y}}{\frac{\partial\zeta(3,\frac{1}{2}+iy)}{\partial y}}$ on $[\frac{1}{2},\frac{\sqrt3}{2}]$}$)$. There is a interior point $y_{0}=0.7035146\cdots$ such that
$\frac{\frac{\partial \zeta(6,\frac{1}{2}+iy)}{\partial y}}{\frac{\partial\zeta(3,\frac{1}{2}+iy)}{\partial y}}$
 is decreasing on $[\frac{1}{2}, y_0]$, and increasing on $[y_0, \frac{\sqrt3}{2}]$.
 Here $y_0$ is characterized by
 \begin{equation}\aligned\nonumber
y_0\;\;\hbox{is the unique minimizer of}\;\;\min_{y\in[\frac{1}{2},\frac{\sqrt3}{2}]}\frac{\frac{\partial}{\partial y}\zeta(6,\frac{1}{2}+iy)}{\frac{\partial}{\partial y}\zeta(3,\frac{1}{2}+iy)}
.
\endaligned\end{equation}

\end{proposition}

\begin{figure}
\centering
 \includegraphics[scale=0.58]{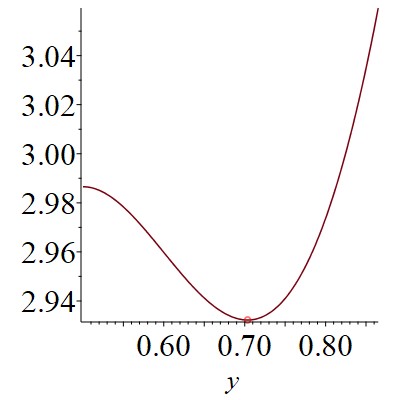}\includegraphics[scale=0.58]{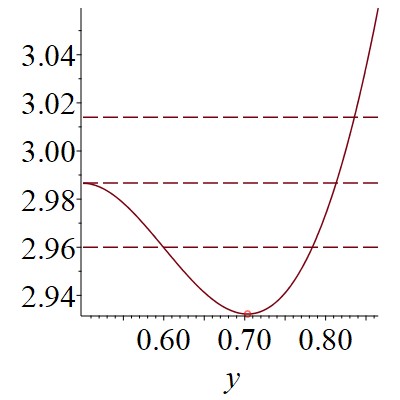}
 \caption{The minimizer $y_0$, shape of $\frac{\frac{\partial \zeta(6,\frac{1}{2}+iy)}{\partial y}}{\frac{\partial\zeta(3,\frac{1}{2}+iy)}{\partial y}}$
 and solutions to $b=\frac{\frac{\partial \zeta(6,\frac{1}{2}+iy)}{\partial y}}{\frac{\partial\zeta(3,\frac{1}{2}+iy)}{\partial y}}$
 }
\label{Y-Shape}
\end{figure}

Note that Proposition \ref{Prop71} follows by Proposition \ref{Prop72} in view of \eqref{J1} and \eqref{J2}.
To prove Proposition \ref{Prop71}, it suffices to prove Proposition \ref{Prop72}.
See Picture \ref{Y-Shape} for Propositions \ref{Prop71} and \ref{Prop72}. To prove Propositions \ref{Prop72}, we do some prepare work.

We cite a useful monotonicity rule in calculus:
\begin{lemma}[{\bf A monotonicity rule \cite{Anderson2006}}]\label{LemmaH5} Let $a<b\in\R$, and let $f,g: [a,b]\rightarrow\R$ be continuous functions
 that are differentiable on $(a,b)$, with $f(a)=g(a)=0$ or $f(b)=g(b)=0$. Assume that $g'(x)\neq0$
for each $x$ in $(a,b)$. If $\frac{f'}{g'}$ is increasing(decreasing) on $(a,b)$, then so is $\frac{f}{g}$.

\end{lemma}

To use Lemma \ref{LemmaH5}, we introduce a preliminary lemma:
\begin{lemma}[{\bf The inflexion point of $\zeta(3,\frac{1}{2}+iy)$}]\label{LemmaH5b}
There exists $\overline{y}_0$ such that
\begin{equation}\aligned\nonumber
\frac{\partial^2}{\partial y^2}\zeta(3,\frac{1}{2}+iy)\begin{cases}
<0, \;\;\;\;\;\;\;\;\;\;\;\; b\in[\frac{1}{2},\overline{y}_0);\\
>0,\;\;\;\;\;\;\;\;\;\;\;\;b\in(\overline{y}_0,\frac{\sqrt3}{2}].
\end{cases}
\endaligned\end{equation}
Numerically, $\overline{y}_0=0.65546688\cdots.$
\end{lemma}
Lemma \ref{LemmaH5b} is deduced by the following
\begin{lemma}\label{LemmaH5b2}
\begin{equation}\aligned\nonumber
\frac{\partial^3}{\partial y^3}\zeta(3,\frac{1}{2}+iy)\geq30>0\;\;\hbox{for}\;\;y\in[\frac{1}{2},\frac{\sqrt3}{2}].
\endaligned\end{equation}
\end{lemma}
Lemma \ref{LemmaH5b2} is proved by direct estimates.

We state a consequence of Lemma \ref{Geee} which is useful in our later estimates.
\begin{lemma}[A functional equation]\label{LemmaH6} $\zeta(6,\frac{1}{2}+iy), \zeta(3,\frac{1}{2}+iy)$ and their quotient $\frac{\frac{\partial \zeta(6,\frac{1}{2}+iy)}{\partial y}}{\frac{\partial\zeta(3,\frac{1}{2}+iy)}{\partial y}}$ satisfy the functional equation
\begin{equation}\aligned\label{HHH}
\mathcal{H}(\frac{1}{4y})=\mathcal{H}(y).
\endaligned\end{equation}
Consequently,
\begin{equation}\aligned\nonumber
\mathcal{H}'(y)=-\frac{1}{4y^2}\mathcal{H}'(\frac{1}{4y})\;\;\hbox{and then}\;\; \mathcal{H}'(\frac{1}{2})=0.
\endaligned\end{equation}
\end{lemma}

By Lemma \ref{LemmaH6}, one has

\begin{lemma}[Evaluating of derivatives at $\frac{1}{2}$]\label{LemmaH7}

\begin{equation}\aligned\nonumber
\Big(\frac{\partial^3\zeta(6,\frac{1}{2}+iy)}{\partial y^3}\frac{\partial^2\zeta(3,\frac{1}{2}+iy)}{\partial y^2}
-\frac{\partial^2\zeta(6,\frac{1}{2}+iy)}{\partial y^2}\frac{\partial^3\zeta(3,\frac{1}{2}+iy)}{\partial y^3}
\Big)\mid_{y=\frac{1}{2}}=0.
\endaligned\end{equation}
\end{lemma}
\begin{proof}{\bf Proof of Lemma \ref{LemmaH7}.}
Since $\zeta(6,\frac{1}{2}+iy), \zeta(3,\frac{1}{2}+iy)$ satisfy the functional equation \eqref{HHH} by Lemma \ref{LemmaH6},
by taking derivatives, one has that $\zeta(6,\frac{1}{2}+iy), \zeta(3,\frac{1}{2}+iy)$ satisfy
\begin{equation}\aligned\label{HHH2}
\mathcal{H}'''(y)=-\frac{1}{64y^4}\mathcal{H}'''(\frac{1}{4y})-\frac{3}{8y^5}\mathcal{H}''(\frac{1}{4y})
-\frac{3}{2y^4}\mathcal{H}'(\frac{1}{4y}).
\endaligned\end{equation}
Substituting $y=\frac{1}{2}$ in \eqref{HHH2}, one gets
\begin{equation}\aligned\nonumber
\mathcal{H}'''(\frac{1}{2})=-6\mathcal{H}''(\frac{1}{2}).
\endaligned\end{equation}
Then
\begin{equation}\aligned\label{HHH3}
\frac{\partial^3\zeta(6,\frac{1}{2}+iy)}{\partial y^3}\mid_{y=\frac{1}{2}}=-6\frac{\partial^2\zeta(6,\frac{1}{2}+iy)}{\partial y^2}\mid_{y=\frac{1}{2}}\\
\frac{\partial^3\zeta(3,\frac{1}{2}+iy)}{\partial y^3}\mid_{y=\frac{1}{2}}=-6\frac{\partial^2\zeta(3,\frac{1}{2}+iy)}{\partial y^2}\mid_{y=\frac{1}{2}}
\endaligned\end{equation}
The conclusion follows by \eqref{HHH3}.
\end{proof}

Denote that
\begin{equation}\aligned\label{NNN}
\mathcal{A}(y):=\frac{\zeta(6,\frac{1}{2}+iy)}{\partial y},\\
\mathcal{B}(y):=\frac{\zeta(3,\frac{1}{2}+iy)}{\partial y}
\endaligned\end{equation}
By the notations in \eqref{NNN}, Proposition \ref{Prop72} equivalents to

\begin{lemma}\label{LemmaH8} The monotonicity of the quotients.
\begin{itemize}
  \item [(1)] $$\Big(\frac{\mathcal{A}(y)}{\mathcal{B}(y)}\Big)'<0\;\;\hbox{for}\;\;y\in(\frac{1}{2},y_0);$$
  \item [(2)] $$\Big(\frac{\mathcal{A}(y)}{\mathcal{B}(y)}\Big)'>0\;\;\hbox{for}\;\;y\in(y_0,\frac{\sqrt3}{2}).$$
\end{itemize}
Here $y_0$ is defined in Proposition \ref{Prop72} and numerically $y_{0}=0.7035146\cdots$.
\end{lemma}

To prove Proposition \ref{Prop72}, it suffices to prove Lemma \ref{LemmaH8}.
We shall divide the proof into three cases as follows:
{\bf case a}: $[\frac{1}{2}, 0.65]$; {\bf case b}: $[0.65,0.76]$; {\bf case c}: $[0.76,\frac{\sqrt3}{2}]$.

Since
\begin{equation}\aligned\nonumber
\Big(\frac{\mathcal{A}(y)}{\mathcal{B}(y)}\Big)'=\frac{\mathcal{A}'(y)\mathcal{B}(y)-\mathcal{A}(y)\mathcal{B}'(y)}{\mathcal{B}^2(y)}
\endaligned\end{equation}
and $y_0$ is a critical point $\Big(\frac{\mathcal{A}(y)}{\mathcal{B}(y)}\Big)$, then
\begin{equation}\aligned\label{AB1}
\Big(\mathcal{A}'(y)\mathcal{B}(y)-\mathcal{A}(y)\mathcal{B}'(y)\Big)\mid_{y=y_0}=0.
\endaligned\end{equation}

The proof of {\bf case b} is based on an elementary identity
\begin{equation}\aligned\label{KKJJ}
\Big(\mathcal{A}'(y)\mathcal{B}(y)-\mathcal{A}(y)\mathcal{B}'(y)\Big)'=\mathcal{A}''(y)\mathcal{B}(y)-\mathcal{A}(y)\mathcal{B}''(y).
\endaligned\end{equation}

We then estimate $\mathcal{A}''(y)\mathcal{B}(y)-\mathcal{A}(y)\mathcal{B}''(y)$ and prove {\bf case b}. We estimate directly that
\begin{lemma}\label{Lemma6576}
For $y\in[0.65,0.76]$,
\begin{equation}\aligned\nonumber
\mathcal{A}''(y)\mathcal{B}(y)-\mathcal{A}(y)\mathcal{B}''(y)\geq6.
\endaligned\end{equation}
\end{lemma}
Lemma \ref{Lemma6576}, \eqref{AB1} and \eqref{KKJJ} yield that
\begin{lemma}[{\bf case b}]
\begin{equation}\aligned\nonumber
\Big(\frac{\mathcal{A}(y)}{\mathcal{B}(y)}\Big)'&>0\;\;\hbox{for}\;\;y\in(y_0,0.76]\\
\Big(\frac{\mathcal{A}(y)}{\mathcal{B}(y)}\Big)'&<0\;\;\hbox{for}\;\;y\in[0.65,y_0).
\endaligned\end{equation}
\end{lemma}

Note that
\begin{equation}\aligned\label{Condition}
\mathcal{A}(\frac{1}{2})=\mathcal{B}(\frac{1}{2})=0,\;\;
\mathcal{A}(\frac{\sqrt3}{2})=\mathcal{B}(\frac{\sqrt3}{2})=0
\endaligned\end{equation}
followed by $\frac{1}{2},\frac{\sqrt3}{2}$ are critical points of $\zeta(s,\frac{1}{2}+iy), s>0$(see e.g. \cite{Ran1953}).
To prove {\bf case c}, we use Lemma \ref{LemmaH5} given by \eqref{Condition}. Namely, we turn to compute that $\Big(\frac{\mathcal{A}'(y)}{\mathcal{B}'(y)}\Big)'$.
Since
\begin{equation}\aligned\nonumber
\Big(\frac{\mathcal{A}'(y)}{\mathcal{B}'(y)}\Big)'=\frac{\mathcal{A}''(y)\mathcal{B}'(y)-\mathcal{A}'(y)\mathcal{B}''(y)}{\mathcal{B'}^2(y)}
\endaligned\end{equation}
The proof of {\bf case c} is based on a direct checking
\begin{lemma}\label{LemmaHc}
For $y\in[0.76,\frac{\sqrt3}{2}]$,
\begin{equation}\aligned\nonumber
\mathcal{A}''(y)\mathcal{B}'(y)-\mathcal{A}'(y)\mathcal{B}''(y)\geq9.
\endaligned\end{equation}
This yields that
\begin{equation}\aligned\nonumber
\Big(\frac{\mathcal{A}'(y)}{\mathcal{B}'(y)}\Big)'>0\;\;\hbox{for}\;\;y\in[0.76,\frac{\sqrt3}{2}].
\endaligned\end{equation}
\end{lemma}
We then prove {\bf case c}.
\begin{lemma}[{\bf case c}]\label{LemmaHc2}
It holds that
\begin{equation}\aligned\nonumber
\Big(\frac{\mathcal{A}(y)}{\mathcal{B}(y)}\Big)'>0\;\;\hbox{for}\;\;y\in[0.76,\frac{\sqrt3}{2}].
\endaligned\end{equation}
\end{lemma}
\begin{proof}{\bf Proof of Lemma \ref{LemmaHc2}.}
By Lemma \ref{LemmaH5b}, ${\mathcal{B}''(y)}>0$ for $y\in[0.76,\frac{\sqrt3}{2}]$. Then the result follows by Lemma \ref{LemmaH5} and Lemma \ref{LemmaHc}.

\end{proof}

 It remains to prove {\bf case a}. The proof is similar to that of {\bf case c} and with one more delicate point.

We estimate directly that
\begin{lemma}[{\bf case a}]\label{LemmaHa}
For $y\in(\frac{1}{2},0.65]$,
\begin{equation}\aligned\nonumber
\mathcal{A}''(y)\mathcal{B}'(y)-\mathcal{A}'(y)\mathcal{B}''(y)<0.
\endaligned\end{equation}
This yields that
\begin{equation}\aligned\nonumber
\Big(\frac{\mathcal{A}'(y)}{\mathcal{B}'(y)}\Big)'<0\;\;\hbox{for}\;\;y\in(\frac{1}{2},0.65].
\endaligned\end{equation}
\end{lemma}
\begin{lemma}[{\bf case a}]\label{LemmaHa}
For $y\in(\frac{1}{2},0.65]$,
\begin{equation}\aligned\nonumber
\mathcal{A}''(y)\mathcal{B}'(y)-\mathcal{A}'(y)\mathcal{B}''(y)<0.
\endaligned\end{equation}
This yields that
\begin{equation}\aligned\nonumber
\Big(\frac{\mathcal{A}'(y)}{\mathcal{B}'(y)}\Big)'<0\;\;\hbox{for}\;\;y\in(\frac{1}{2},0.65].
\endaligned\end{equation}
\end{lemma}

Similar to the proof of {\bf case c}, by Lemma \ref{LemmaHa}, it follows that
 \begin{equation}\aligned\nonumber
\Big(\frac{\mathcal{A}(y)}{\mathcal{B}(y)}\Big)'<0\;\;\hbox{for}\;\;y\in(\frac{1}{2},0.65].
\endaligned\end{equation}

It remains to prove Lemma \ref{LemmaHa}.

Note that by Lemma \ref{LemmaH7} and notations in \eqref{NNN}, one has
\begin{equation}\aligned\label{NNN0}
\Big(\mathcal{A}''(y)\mathcal{B}'(y)-\mathcal{A}'(y)\mathcal{B}''(y)\Big)\mid_{y=\frac{1}{2}}=0.
\endaligned\end{equation}
In view of \eqref{NNN0}, to prove Lemma \ref{LemmaHa}, we shall further divide the {\bf case a} into two subcases,
 namely {\bf case a$_1$:} $(\frac{1}{2},0.54]$, {\bf case a$_2$:} $[0.54,0.65]$.

For {\bf case a$_1$}, since $$\mathcal{A}'''(y)\mathcal{B}'(y)-\mathcal{A}'(y)\mathcal{B}'''(y)=\Big(\mathcal{A}''(y)\mathcal{B}'(y)-\mathcal{A}'(y)\mathcal{B}''(y)\Big)'$$ one estimates directly that

\begin{lemma}[{\bf case a$_1$}]\label{LemmaHa1}
For $y\in(\frac{1}{2},0.54]$,
\begin{equation}\aligned\nonumber
\mathcal{A}'''(y)\mathcal{B}'(y)-\mathcal{A}'(y)\mathcal{B}'''(y)\leq-1600.
\endaligned\end{equation}
This and \eqref{NNN0} yield that
\begin{equation}\aligned\nonumber
\mathcal{A}''(y)\mathcal{B}'(y)-\mathcal{A}'(y)\mathcal{B}''(y)<0\;\;\hbox{for}\;\;y\in(\frac{1}{2},0.54].
\endaligned\end{equation}
\end{lemma}
For {\bf case a$_2$}, we estimate directly that
\begin{lemma}[{\bf case a$_2$}]\label{LemmaHa2}
For $y\in[0.54,0.65]$,
\begin{equation}\aligned\nonumber
\mathcal{A}''(y)\mathcal{B}'(y)-\mathcal{A}'(y)\mathcal{B}''(y)\leq-90.
\endaligned\end{equation}
\end{lemma}

Lemma \ref{LemmaHa} follows by Lemmas \ref{LemmaHa1} and \ref{LemmaHa2}.

\section{Proof of Theorem \ref{Th1}}

Recall that
\begin{equation}\aligned\nonumber
\Gamma_a&=\{
z\in\mathbb{H}: \Re(z)=0,\; \Im(z)\geq1
\};\\
\Gamma_b&=\{
z\in\mathbb{H}: |z|=1,\; \Im(z)\in[\frac{\sqrt3}{2},1]
\}.
\endaligned\end{equation}
And
\begin{equation}\aligned\nonumber
\Gamma=\Gamma_a\cup\Gamma_b.
\endaligned\end{equation}
By Theorems \ref{ThA}, \ref{Th102} and \ref{Th112}, we obtain that
\begin{theorem}[]
For $b\in\R$,
\begin{equation}\aligned\nonumber
\min_{z\in\mathbb{H}}\Big(\zeta(6,z)-b\zeta(3,z)\Big)=\min_{z\in\overline{\mathcal{D}_{\mathcal{G}}}}\Big(\zeta(6,z)-b\zeta(3,z)\Big)
=\min_{z\in\Gamma}\Big(\zeta(6,z)-b\zeta(3,z)\Big)
.
\endaligned\end{equation}
\end{theorem}
See Picture \ref{PFFF}.
Then Theorem \ref{Th1} follows by Theorems \ref{Thaa} and \ref{Thbb} in Sections \ref{section6} and \ref{section7} respectively.

\bigskip
\bigskip
\bigskip
\noindent
{\bf Acknowledgements.}
 The research of S. Luo is partially supported by NSFC(Nos. 12261045, 12001253) and double thousands plan of Jiangxi(jxsq2019101048). The research of J. Wei is partially supported by NSERC of Canada.

\section{Appendix}

On stating Lemmas \ref{LemmaH5b2}, \ref{Lemma6576}, \ref{LemmaHc}, \ref{LemmaHa1}, \ref{LemmaHa2},
\ref{Lemmag54}, \ref{Lemma434}, \ref{LemmaA10}, \ref{LemmaA11}, \ref{Lemma319}, we use the method by a direct estimate or checking or computation. We state how the
direct estimate or checking or computation are and why work, the details of the computation are omitted.

Indeed, they are on finite and small intervals with uniform upper/lower bounds. The function is one variable and is exponential decaying by Proposition \ref{PropW}, and can be approached by finite terms(in fact only 2 or 3 or 4 terms is enough in these estimates), the error is very small comparing the uniform upper/lower bounds, the computation are tedious and we omit here.

\end{document}